\documentclass[11pt]{amsart}
\usepackage{amsmath, amssymb, amsthm}
\usepackage[margin=1in]{geometry}
\usepackage{esint}
\usepackage{xcolor}
\usepackage{hyperref}
\hypersetup{colorlinks = true, linkbordercolor = red}

\usepackage[pdftex]{graphicx}
\pdfoutput=1

\author{Nikola Kamburov}

\address{Nikola Kamburov, Facultad de Matem\'aticas, Pontificia Universidad Cat\'olica de Chile, Avenida Vicu\~na Mackenna 4860, Santiago 7820436, Chile}
\email{nikamburov@mat.uc.cl}

\author{Luciano Sciaraffia}

\address{Luciano Sciaraffia, Facultad de Matem\'aticas, Pontificia Universidad Cat\'olica de Chile, Avenida Vicu\~na Mackenna 4860, Santiago 7820436, Chile}
\email{lvsciaraffia@uc.cl}

\thanks{The first author (NK) was partially supported by Proyecto FONDECYT Iniciaci\'on No. 11160981.}

\title[Nontrivial solutions to Serrin's problem]{Nontrivial solutions to Serrin's problem in annular domains}


\newtheorem{theorem}{Theorem}[section]
\newtheorem{prop}[theorem]{Proposition}
\newtheorem{lemma}[theorem]{Lemma}
\newtheorem{corollary}[theorem]{Corollary}
\newtheorem{remark}{Remark}[section]

\numberwithin{equation}{section}

\def\R{\mathbb{R}}

\def\Z{\mathbb{Z}}
\def\N{\mathbb{N}}

\def\S{\mathbb{S}}

\def\e{\varepsilon}
\def\m{\mu}
\def\n{\nu}
\def\G{\Gamma}
\def\D{\Delta}
\def\O{\Omega}
\def\a{\alpha}

\def\g{\gamma}
\def\d{\delta}
\def\s{\sigma}
\def\th{\theta}
\def\f{\phi}

\def\l{\lambda}

\def\F{\Phi}
\def\o{\omega}

\def\p{\partial}

\def\disp{\displaystyle}

\def\id{\mathrm{id} \,}
\def\im{\mathrm{im} \,}

\renewcommand{\vec}{\mathbf}

\def\de{\partial}

\def\XXint#1#2#3{{\setbox0=\hbox{$#1{#2#3}{\int}$ }
\vcenter{\hbox{$#2#3$ }}\kern-.6\wd0}}

\newcommand{\Nz}{\mathbb{N}_0}

\begin{document}

\begin{abstract}
	We construct nontrivial smooth bounded domains $\O \subseteq \R^n$ of the form $\O_0 \setminus \overline{\O}_1$, bifurcating from annuli, for which there exists a positive solution to the overdetermined boundary value problem 
	$$\begin{cases}
        -\D u = 1 ,\; u > 0  & \text{ in } \O, \\
        u = 0 ,\; \p_\n u = \text{const} & \text{ on } \p\O_0, \\
        u = \text{const} ,\; \p_\n u = \text{const} & \text{ on } \p\O_1,
    \end{cases}$$
	where $\n$ stands for the inner unit normal to $\p\O$. From results by Reichel \cite{reichel1995radial} and later by Sirakov \cite{sirakov2001symmetry}, it is known that the condition $\p_\n u \leq 0$ on $\p\O_1$ is sufficient for rigidity to hold, namely, the only domains which admit such a solution are annuli and solutions are radially symmetric. Our construction shows that the condition is also necessary. Furthemore, we show that the constructed domains are self-Cheeger.

\end{abstract}

\bibliographystyle{alpha}

\maketitle

\section{Introduction and main result}\label{intro}

Let $\O\subseteq \R^n$ be a bounded, \emph{connected} $C^2$-domain of the form $\O = \O_0\setminus \overline{\O}_1$, where $\O_0$ and $\O_1$ are bounded domains in $\R^n$, $n\geq 2$, with $\O_1\Subset \O_0$. The present paper is devoted to the overtermined boundary value problem
\begin{equation}\label{eq:prob1}
\begin{cases}
    -\D u = 1   & \text{ in } \O , \\
    u = 0 ,\; \p_\n u = c_0 & \text{ on } \p\O_0 , \\
    u = a ,\; \p_\n u = c_1 & \text{ on } \p\O_1 ,
\end{cases}
\end{equation}
where $\n$ denotes the \emph{inner} unit normal to $\O$ and $a\geq 0$, $c_0$ and $c_1$ are real constants. Note that whenever \eqref{eq:prob1} admits a solution $u\in C^2(\overline{\O})$, then $u$ is strictly positive in $\O$ due the strong maximum principle, while the Hopf lemma implies that the constant $c_0>0$.

Problem \eqref{eq:prob1} arises in classical models in the theory of elasticity, fluid mechanics and electrostatics -- see \cite{sirakov2002overdetermined} for a discussion of applications. The special case in which $a=0$ and $c_0=c_1$ was treated by Serrin in his seminal 1971 paper \cite{serrin1971symmetry}. He showed that a strong property of \emph{rigidity} was forced upon any solution $u\in C^2(\overline{\O})$ and upon the shape of the domain $\O$ supporting it: namely, $\O$ has to be a ball ($\O_1 = \emptyset$) and $u$ has to be radially symmetric and monotonically decreasing along the radius. Serrin's proof is based on the \emph{moving planes method},  pioneered earlier by Alexandrov \cite{alexandrov1962characteristic} in a geometric context, and has ever since been a powerful tool for establishing symmetry of positive solutions to second-order elliptic problems. An important artifact of the method is that proving radial symmetry comes hand in hand with proving the monotonicity of solutions in the radial direction.

Reichel \cite{reichel1995radial} adapted Serrin's method to analyze \eqref{eq:prob1} in the case when $$u|_{\de \O_1} = a>0 \quad\text{and}\quad u_{\nu}|_{\de \O_1}=c_1\leq 0.$$ 
Under the additional assumption that $0<u<a$ in $\O$, he showed that $u$ has to be radially symmetric and the domain $\O$ -- a standard annulus. Several years later, Sirakov \cite{sirakov2001symmetry} removed the extra assumption and proved a more general rigidity theorem, allowing for separate constant Dirichlet conditions $u|_{\g_i} = a_i>0$ and separate constant Neumann conditions $u_{\nu}|_{\g_i}=c_i\leq 0$ to be imposed on each connected component $\g_i$ of the inner boundary $\de \O_1$, $i=1,\ldots, l$. The assumption of non-positivity of each Neumann condition $u_{\nu}|_{\g_i}=c_i\leq 0$ is crucial for the moving planes method to run and yield the radial symmetry and monotonicity of solutions. Furthemore, just like Serrin's result, the rigidity theorems by Reichel and Sirakov apply to a more general class of second order elliptic equations of which \eqref{eq:prob1} in an important special case. In $n=2$ dimensions, a symmetry result for \eqref{eq:prob1} was obtained earlier by Willms, Gladwell and Siegel \cite{willms1994symmetry} under some additional boundary curvature assumptions. See also \cite{khavinson2005overdetermined} for a complex analytic approach to \eqref{eq:prob1} when $n=2$.

In this paper we focus on a case, in which the Neumann condition on the \emph{inner} boundary is positive:
\begin{equation}\label{BC}
u|_{\de \O_0} = 0, \quad u|_{\de \O_1} = a>0 \quad\text{and} \quad u_{\nu}|_{\de \O_0}= u_{\nu}|_{\de \O_1}=c>0.
\end{equation}
We immediately notice that over annuli $\O$, \eqref{eq:prob1} now possesses radial solutions that are \emph{not monotone} along the radius -- see Lemma \ref{trivialsol} for the formula of a one-parameter family of such examples. Their presence hints that proving radial symmetry for solutions of \eqref{eq:prob1} would be out of the scope of the moving planes method. Thus, one is led to conjecture that, under \eqref{BC}, \emph{radial rigidity does not hold} for solutions of \eqref{eq:prob1}.

Our main result confirms that this is indeed the case. 

\begin{theorem}\label{teo1}
    There exist smooth bounded annular domains of the form $\O=\O_0 \setminus \overline{\O}_1 \subseteq \R^n$, which are different from standard annuli, and positive constants $a$ and $c$, for which the overdetermined problem \eqref{eq:prob1} admits a solution $u \in C^{\infty}(\overline{\O})$ satisfying \eqref{BC}.
\end{theorem}

We construct these nontrivial annular domains and their corresponding solutions by the means of the Crandall-Rabinowitz Bifurcation Theorem \cite{crandall1971bifurcation}. In this manner, we actually obtain a smooth branch of domains and solutions (in fact, a whole sequence of distinct branches) bifurcating from the trivial branch of standard annuli admitting the radial, non-monotone solutions of \eqref{eq:prob1}, mentioned above. This is more precisely stated in the body of Theorem \ref{teo2} in Section \ref{prelim}, of which Theorem \ref{teo1} is an immediate corollary.

The overdetermined problem \eqref{eq:prob1} with boundary conditions \eqref{BC} has a connection to the so-called \emph{Cheeger problem}: given a bounded domain $\O\subseteq \R^n$, find 
\begin{equation}\label{eq:cheeger}
h_1(\O):=\inf\{P(E)/|E|: E\subseteq \O\}
\end{equation} 
where $|E|$ is the Lebesgue measure of $E$ and $P(E)$ denotes the perimeter functional (see \cite{giusti1984}). The constant $h_1(\O)$ is known as the \emph{Cheeger constant} for the domain $\O$, and a subset $E\subseteq \O$, for which the infimum in \eqref{eq:cheeger} is attained, is called a \emph{Cheeger set} for $\O$. See the surveys \cite{parini2011, leonardi2015} for an overview of the Cheeger problem and a discussion of applications. It turns out that the domains $\O$ constructed in Theorem \ref{teo1} are precisely their own Cheeger sets. Such domains are called \emph{self-Cheeger}.

\begin{corollary}\label{coro:cheeger} Let $\O$ be any one of the smooth annular domains in Theorem \ref{teo1} that admits a solution $u\in C^{\infty}(\overline{\O})$ of \eqref{eq:prob1}-\eqref{BC}. Then
\[
h_1(\O) =\frac{P(\O)}{|\O|} =1/c 
\]
and $\O$ is the unique minimizer of the Cheeger problem \eqref{eq:cheeger}.
\end{corollary}
In this manner, Corollary \ref{coro:cheeger} establishes the existence of non-radially symmetric domains that are self-Cheeger.

Constructions of non-trivial solutions to overdetermined elliptic problems have been prominent in the literature in recent years. Many of them have been driven by a famous conjecture of Berestycki, Caffarelli and Nirenberg \cite{berestycki1997monotonicity}, according to which, if $f$ is a Lipschitz function and $\O\subseteq \R^n$ is an unbounded smooth domain, such that $\R^n \setminus \overline{\O}$ is connected, then the overdetermined problem
\begin{equation}\label{exterioroverdetermined}
\begin{cases}
    \D u + f(u) = 0 ,\; u > 0  & \text{ in } \O, \\
    u = 0 ,\; \p_\n u = \text{const} & \text{ on } \p\O,
\end{cases}
\end{equation}
admits a positive bounded solution if and only if $\O$ is a half space, a cylinder or the complement of a ball. First,  Sicbaldi \cite{sicbaldi2010new} constructed domain counterexamples $\O$ to the conjecture when $n\geq 3$ and $f(u)=\l u$, which bifurcate from cylinders $B^{n-1}\times \R$ for appropriate $\l>0$. Then  Ros, Ruiz and Sicbaldi  \cite{ros2016solutions} constructed a different set of counterexamples for all dimensions $n\geq 2$ and $f(u) = u^p-u$, $p>1$, that bifurcate from the complement of a ball $\R^n\setminus \overline{B_{\l}^n}$. The main tool behind the two results is a bifurcation theorem by Krasnoselski (see \cite{kielhofer2011bifurcation}) that is based on topological degree theory and which yields a sequence of domains $\O$, rather than a smooth branch. Schlenk and Sicbaldi \cite{schlenk2012bifurcating} managed to strengthen the construction in \cite{sicbaldi2010new} through the use of the Crandall-Rabinowitz Bifurcation Theorem to obtain a smooth branch of perturbed cylinders $\O$. A similar approach leading to perturbed generalized cylinders was pursued by Fall, Minlend and Weth \cite{fall2017unbounded} in the case of our interest, $f\equiv 1$. The bifurcation method has also been successful in finding nontrivial solutions in versions of \eqref{exterioroverdetermined} set in Riemannian manifolds \cite{morabito2016delaunay, fall2018serrin}.  For other methods of solution construction in overdetermined elliptic problems, we refer to \cite{pacard2009extremal, DS, HHP, kamburov2013free, KhavLundTeo, Traizet, del2015serrin, delay2015extremal, fall2015serrin, LWW, JP} among others.

Our approach to Theorem \ref{teo1} is aligned with that of Schlenk and Sicbaldi in \cite{schlenk2012bifurcating} and Fall, Minlend and Weth in \cite{fall2017unbounded, fall2018serrin}: we translate the problem to a non-linear, non-local operator equation in appropriate function spaces and we derive the necessary spectral and Fredholm tranversality properties of the linearized operators in order to implement the Crandall-Rabinowitz Bifurcation Theorem. However, unlike the quoted results above where symmetry considerations allow the authors to perturb all connected boundary components of the underlying domains in the same symmetric fashion, we are bound by the geometry of the standard annulus \mbox{$\O_{\l}=\{x\in \R^n: \l <|x|<1\}$} to deform its two \emph{non-symmetric} boundary components differently. Thus, the function spaces that we work in are necessarily product spaces of two factors that correspond to the two separate connected components of $\de \O_{\l}$. This is a new feature, and as far as we are aware, our treatment is the first one in this line of results that deals with it. 

The paper is organized as follows. In the next section we outline the strategy of the construction leading up to the statement of Theorem \ref{teo2}, which refines Theorem \ref{teo1}, and we show how the latter follows from the former. In Section \ref{probset} we set up the problem as an operator equation $F_{\l}(\vec{v})=\vec{0}$, where $F_\l: U \subseteq (C^{2,\a}(\S^{n-1}))^2 \to (C^{1,\a}(\S^{n-1}))^2$, and compute a formula for its linearization $L_{\l}:=D_{\vec{v}} F_\l|_{\vec{v}=\vec{0}}$ in terms of the Dirichlet-to-Neumann operator for the Laplacian in $\O_\l$ (Proposition \ref{prop:linearization}). In Section \ref{opspec} we study the spectrum of $L_{\l}$: we find that for each $k\in \Nz=\N\cup\{0\}$, $L_{\l}$ has \emph{two} different eigenvalue branches $\mu_{k,1}(\l)<\mu_{k,2}(\l)$ with associated eigenvectors in the subspaces $\R \mathcal{Y}_k \oplus \R \mathcal{Y}_k$, where $\mathcal{Y}_k$ is any spherical harmonic of degree $k$. In Lemmas \ref{eigenvalsdecrease} and \ref{eigenvalsincrease} we establish key monotonicity properties for $\mu_{k,1}(\l)$ and $\mu_{k,2}(\l)$ in both $\l$ and $k$, from which we infer that for $k\geq 2$ the first eigenvalue branch $\m_{k,1}(\l)$ is strictly decreasing in $\l\in (0,1)$, changing sign once, while the second $\m_{k,2}(\l)>0$ is always positive. In Section \ref{proofcomp} we restrict the operators to pairs of $G$-invariant functions on the sphere for an appropriate group of isometries $G$ so as to ensure that, whenever $0$ is an eigenvalue of the restricted $L_{\l}$, it is simple. We then verify the relevant Fredholm mapping properties (Proposition \ref{verificationLlambda}), necessary to apply the Crandall-Rabinowitz Bifurcation Theorem and complete the proof of Theorem \ref{teo2}. Finally, in Section \ref{sec:coro} we provide the proof of Corollary \ref{coro:cheeger}.

\section*{Acknowledgements.}
The authors would like to thank Dimiter Vassilev for bringing up Serrin's problem in annular domains to their attention.

\section{Outline of strategy and refinement of the main theorem}\label{prelim}

Let us first introduce some notation. For any $\lambda \in (0,1)$ we denote the standard annulus of inner radius $\l$ and outer radius $1$ by
\begin{align*}
    \O_\l & := \left\{ x \in \R^n: \l < |x| < 1 \right\}  
\end{align*}
and let its two boundary components be
\begin{align*}
         \G_1 & := \left\{ x \in \R^n: |x| = 1 \right\} = \S^{n-1} , \\
    \G_\l & := \left\{ x \in \R^n: |x| = \l \right\} = \l \S^{n-1} ,
\end{align*}
where $\S^{n-1}$ is the unit sphere in $\R^n$, centered at the origin. 

We will construct the nontrivial solutions $u$ and domains $\O$ solving \eqref{eq:prob1}-\eqref{BC} by bifurcating away, at certain critical values of the bifurcation parameter $\l$, from the branch of non-monotone radial solutions $u_{\l}$ of \eqref{eq:prob1} defined on the annuli $\O_\l$, for which $\de_\nu u_{\l}=c_\l$ is the same constant on all of $\de \O_\l$. We describe this branch of solutions explicitly in the lemma below. 

\begin{lemma}\label{trivialsol}
    For each $\l \in (0,1)$, there exist unique positive values $a_\l$ and $c_\l$ given by
    \begin{align}\label{eq:bdrydata}
        a_\l & = \begin{cases} \disp \frac{1}{2} \l \log\l + \frac{1}{4} \left( 1-\l^2 \right) & \text{ if } n=2, \\ \disp \frac{1}{n}\l \frac{\l^{n-2}-1}{n-2}\frac{1+\l}{1+\l^{n-1}} + \frac{1}{2n} \left( 1-\l^2 \right) & \text{ if } n \geq 3, \end{cases} \\
        c_\l & = \frac{1}{n}\frac{1-\l^{n}}{1+\l^{n-1}} 
    \end{align}
    such that for $\O = \O_\lambda$ the problem \eqref{eq:prob1} has a unique positive solution $u=u_\l$ with boundary conditions
    $$ u = 0 \text{ on } \G_1, \qquad u = a_\l \text{ on } \G_\l, \qquad \frac{\p u}{\p\n} = c_\l \text{ on } \p\O_\l. $$
    The solution is radially symmetric and given by
    \begin{equation}\label{eq:trivialsol}
        u_\l(x) =
            \begin{cases}
                \disp \frac{1}{2}\l\log |x| + \frac{1}{4} \left( 1 -|x|^2 \right), & \text{ if } n = 2, \\
                \disp \frac{\lambda^{n-1}}{n(n-2)}\frac{1+\lambda}{1+\lambda^{n-1}} \left( 1-|x|^{2-n} \right) + \frac{1}{2n} \left( 1-|x|^2 \right) & \text{ if } n \geq 3.
            \end{cases}
    \end{equation}
    
    \begin{proof}
        If $u=u(r)$, where $r=|x|$, is a radially symmetric solution to \eqref{eq:prob1}, then $u$ satisfies the ODE        $$u'' + \frac{n-1}{r}u' = -1 ,$$
        where the prime denotes differentiation with respect to $r$. Then simple integration yields
        \begin{equation}\label{eq:firstint}
            u'(r) = \frac{C}{r^{n-1}} -\frac{r}{n}.
        \end{equation}
        Solving $-u'(1) = u'(\l)$ for $C$, we obtain $$C = \frac{1+\l}{n(1+\l^{1-n})}.$$ The formulas \eqref{eq:bdrydata}-\eqref{eq:trivialsol} for $a_\l$, $c_\l$ and $u_{\l}$ follow by integrating \eqref{eq:firstint} once again and setting $u(1)=0$. It remains to confirm that $a_\l>0$ when $\l\in(0,1)$. When $n=2$, this follows from the observation that 
        $$
      \lim_{\l \to 1^-}\frac{da_\l}{d\l}= 0 = \lim_{\l \to 1^-} a_\l \quad \text{and} \quad \frac{d^2 a_{\l}}{d\l^2}= (\l^{-1} - 1)/2 >0 \quad \text{for }  \l\in(0,1).
        $$ 
For $n\geq 3$, we can rewrite the expression \eqref{eq:bdrydata} for $a_{\l}$ as
        \[
        a_\l = \frac{1+\l}{2n(n-2)(1+\l^{n-1})}g(\l) \quad \text{where} \quad g(\l) = (n-2)- n\l - (n-2)\l^n +n\l^{n-1}.
        \]
       Then the positivity of $g(\l)$, and thus of $a_\l$, over $\l\in (0,1)$, follows from the fact that 
        $$\lim_{\l \to 1^-}\frac{dg(\l)}{d\l}= 0 = \lim_{\l \to 1^-} g(\l)  \quad \text{and} \quad \frac{d^2 g(\l)}{d\l^2}= n(n-1)(n-2)\l^{n-3} (1-\l)>0 \quad \text{for } \l\in(0,1).$$
    \end{proof}
\end{lemma}

We will be perturbing the boundary of each annulus $\O_{\l}$ in the direction of the \emph{inner} unit normal to $\de \O_\l$. For a pair of functions $\vec{v} = (v_1,v_2) \in C^{2,\a}(\S^{n-1}) \times C^{2,\a}(\S^{n-1})$ of sufficiently small $C^{2,\a}$-norm, $0<\a<1$, denote the $\vec{v}$-deformation of $\O_\l$ by: 
\begin{align*}
    \O_{\lambda}^{\vec{v}} & := \left\{ x \in \R^n : \l + v_1(x/|x|) < |x| < 1-v_2(x/|x|)  \right\},
\end{align*}
so that its boundary $\p\O_{\l}^{\vec{v}} = \G_{1}^{\vec{v}} \sqcup \G_{\l}^{\vec{v}}$, where
\begin{align*}
    \G_{1}^{\vec{v}} & := \left\{ x \in \R^n : |x| = 1-v_2(x/|x|)  \right\} , \\
      \G_{\lambda}^{\vec{v}} & := \left\{ x \in \R^n : |x| = \l + v_1(x/|x|)  \right\}.
\end{align*}

Our perturbations $\vec{v}$ will ultimately be taken to be invariant with respect to the action of a certain subgroup $G$ of the orthogonal group $O(n)$. Recall that a function $\psi: \O \to \R$, defined on a $G$-invariant domain $\O$, is called $G$-invariant if
\begin{equation*}
    \psi = \psi \circ g, \qquad \text{for every } g \in G.
\end{equation*}
The notation for the usual H\"older and Sobolev function spaces, restricted to $G$-invariant functions, will include a subscript-$G$, as in $C^{k,\a}_G$, $L^2_G$, $H^k_G$, etc.

We know that for each $\l\in (0,1)$ and every $\vec{v}\in (C^{2,\a}(\S^{n-1}))^2$ of appropriately small $C^{2,\a}$-norm, the Dirichlet problem in the perturbed annulus $\O_{\l}^{\vec{v}}$
    \begin{equation}\label{eq:prob2}
        \begin{cases}
            -\D u = 1 & \text{ in } \O_{\l}^{\vec{v}} , \\
            u = 0 & \text{ on } \G_{1}^{\vec{v}} , \\
            u = a_\lambda & \text{ on } \G_{\l}^{\vec{v}} ,
        \end{cases}
    \end{equation}
with $a_{\l}$ defined as in \eqref{eq:bdrydata}, has a unique solution $u = u_\lambda(\vec{v}) \in C^{2,\a}(\overline{\O}_{\l}^{\vec{v}})$. Moreover, $u_\lambda(\vec{0}) = u_\l$, the map $(\vec{v},\l) \mapsto u_\lambda(\vec{v})$ is smooth by standard regularity theory, and if $\vec{v}$ is $G$-invariant, then so are $u_\lambda(\vec{v})$ and $\de_\nu u_\lambda(\vec{v})$, by the uniqueness of solutions. 


We would like to find out when the Dirichlet problem solution $u_{\l}(\vec{v})$ also satisfies a constant Neumann condition on $\de\O_{\l}^{\vec{v}}$. For the purpose, given $\l \in (0,1)$ and 
$U \subseteq C^{2,\a}(\S^{n-1}) \times C^{2,\a}(\S^{n-1})$ -- a sufficiently small neighbourhood of $\vec{0}$, we define the operator 
\begin{align}\label{eq:defFlambda}
	F_\l &:  U  \to C^{1,\a}(\S^{n-1}) \times C^{1,\a}(\S^{n-1}) , \notag \\
	F_\l(\vec{v}) &:=  \frac{1}{c_\l}\left( 
    \left.\frac{\p u_\l(\vec{v})}{\p\n} \right|_{\G_{\l}^{\vec{v}}} - c_\l, \left.\frac{\p u_\l(\vec{v})}{\p\n} \right|_{\G_{1}^{\vec{v}}} -c_\l
 \right) ,
\end{align}
where we canonically identify functions on $\de \O_\l^{\vec{v}}$ with pairs of functions on $\S^{n-1}$.
Schauder theory implies that \eqref{eq:defFlambda} is a good definition and the factor of $1/c_\l$ provides a convenient normalization. Now, the solution $u_{\l}(\vec{v})$ of the Dirichlet problem \eqref{eq:prob2} in $\O_{\l}^{\vec{v}}$ solves the full overdetermined problem \eqref{eq:prob1}-\eqref{BC} if and only if
\begin{equation}\label{eq:functionaleq}
F_{\l}(\vec{v}) = \vec{0}.
\end{equation}
Note that $F_\l(\vec{0}) = \vec{0}$ for every $\l \in (0,1)$. Our goal is to find a branch of solutions $(\vec{v}, \l)$ of \eqref{eq:functionaleq}, bifurcating from this trivial branch $(\vec{0},\l)$. To achieve it, we will need to understand how the kernel of the linearization $L_{\l}:=D_{\vec{v}}F_\l|_{\vec{v}=\vec{0}}$ depends on $\l$. In Proposition \ref{prop:linearization} of the next section we will derive a working formula for $L_{\l}$:
\[
L_\l(w_1,w_2) = \left( \left. -\frac{\p\f_\vec{w}}{\p\n} \right|_{\G_{\l}}  +  \left. \frac{w_1}{c_\l} \frac{\p^2u_\l}{\p r^2} \right|_{\G_{\l}}, \left. -\frac{\p\f_\vec{w}}{\p\n} \right|_{\G_{1}}  +  \left. \frac{w_2}{c_\l} \frac{\p^2u_\l}{\p r^2} \right|_{\G_{1}} \right) ,
\]
where $\phi_\vec{w}$ is the harmonic function $\phi$ on $\O_\l$ with boundary values $\f|_{\G_\l}(x) = w_1(x/|x|)$ and  $\f|_{\G_1}(x) = w_2(x/|x|)$. 

In order to study the kernel of $L_\l$, we will look more generally at the eigenvalue problem
\begin{equation*}
L_{\l}(\vec{w}) = \mu(\l) \vec{w} , \qquad  \vec{w}\in(C^{2,\a}(\S^{n-1}))^2.
\end{equation*}
For each spherical harmonic $\mathcal{Y}_k$ of degree $k\in \Nz$, the subspace $W_k=\text{Span}\{(\mathcal{Y}_k, 0), (0, \mathcal{Y}_k)\}$ is invariant under $L_\l$ and decomposes into eigenspaces for $L_\l|_{W_k}$, associated with two distinct eigenvalues $\mu_{k,1}(\l)<\mu_{k,2}(\l)$, for which we will calculate explicit formulas in Section \ref{opspec}. We will study the dependence of these eigenvalues on both $\l\in (0,1)$ and $k\in \Nz$, focussing on whether they cross $0$ as $\l$ varies in $(0,1)$. It turns out that when $k=1$, $\mu_{1,1}(\l) < 0$ while $\mu_{1,2}(\l) = 0$ for all $\l$, i.e. the eigenspace correponding to $\mu_{1,2}$ is always in the kernel of $L_\l$ (see Remark \ref{rem:k1evals}). This part of $\ker L_\l$ comes from the deformations of $\O_\l$, generated by translations. What we find for $k\geq 2$ is that the first eigenvalue branch $\mu_{k,1}(\l)$ is \emph{strictly decreasing} in $\l$ and that it crosses $0$ at a unique $\l_k^*\in (0,1)$. This is done in Lemma \ref{eigenvalsdecrease}. Moreover, we establish in Lemma \ref{eigenvalsincrease} that both $\mu_{k,1}(\l)$ and $\mu_{k,2}(\l)$ increase \emph{strictly} in $k\in \Nz$ for fixed $\l$. This means that for $k\geq 2$,  the second eigenvalue branch $\mu_{k,2}(\l)>0$ never crosses $0$, while the values of $\l_k^*$, at which  $\mu_{k,1}(\l)$ is zero, form a strictly increasing $k$-sequence; in addition, the eigenvalues for $k=0$, $\mu_{0,1}(\l)<\mu_{0,2}(\l) <0$ (Proposition \ref{prop:collectmono}). Therefore, at the critical values $\l=\l_k^*$, $k \geq 2$, the kernel of $L_{\l_k^*}$ over $(C^{2,\a}(\S^{n-1}))^2$ consists of the $\mu_{k,1}(\l_k^*)$-eigenspace plus the always present component of the $\mu_{1,2} \equiv 0$ eigenspace. 

We would like to point out that the arguments behind proving the key eigenvalue monotonicity properties in Lemmas \ref{eigenvalsdecrease} and \ref{eigenvalsincrease} are of a different nature from the ones used at the analogous phase of the constructions in \cite{schlenk2012bifurcating, fall2017unbounded, fall2018serrin}. The first two papers employed an argument based on rescaling and Bessel function identities, while the third utilized a neat characterization of the eigenvalues of the linearized operator in terms of solutions to a mixed boundary value problem for a first-order ODE. We are afforded neither method in our setting because of the lack of symmetry between the two boundary components of $\de \O_\l$. 
We prove the $\l$-monotonicity of the first branch $\mu_{k,1}(\l)$, $k\geq 2$, by analyzing the explicit formula for $\mu_{k,1}(\l)$ directly and using some delicate estimates, based on hyperbolic trigonometric identities (see the proof of Lemma \ref{eigenvalsdecrease}).  Unfortunately, this approach does not extend to the second eigenvalue branch $\mu_{k,2}(\l)$, $k\geq 2$, which we also believe to be decreasing in $\l$, based on numerics. Showing the latter is ultimately not necessary, since we prove that $\mu_{k,2}(\l)>0$  never contributes to the kernel of $L_\l$. Neither do the eigenvalues branches for $k=0$, $\m_{0,j}(\l)$, $j=1,2$, which are shown to be strictly negative.  In order to establish the $k$-monotonicity of $\mu_{k,j}(\l)$, for fixed $\l$, $j=1,2$, we treat $k$ as a continuous positive variable and extend the $\mu_{k,j}(\l)$ to be smooth functions in $(k,\l)\in (0,\infty)\times(0,1)$, continuous up to $k=0$ (see Remark \ref{rem:analytic}). Even so, trying to prove  $\de_k \mu_{k,j}(\l) >0$ directly from the formula for the eigenvalue turns out to be unyielding, and we accomplish it instead by looking at the matrix representation $M_{\l,k}$ of $L_\l|_{W_k}$ and showing that $\de_k M_{\l,k}$ is positive definite (see the proof of Lemma \ref{eigenvalsincrease}).

The Crandall-Rabinowitz Bifurcation Theorem (see the Appendix) requires that the linearized operator has a one-dimensional kernel at bifurcation values. To achieve this, in Section \ref{proofcomp} we will restrict  the operators $F_\l$ and $L_\l$ to $G$-invariant functions in $(C^{2,\a}(\S^{n-1}))^2$. The symmetry group $G$ will be chosen so as to completely eliminate the eigenspaces of $L_\l$ corresponding to $k=1$ (which contain the $\mu_{1,2}=0$ component of ker $L_\l$), and ensure that, whenever $\mu_{k,1}(\l)$ is an eigenvalue of the restricted $L_\l$ for some $k \geq  2$, it is of multiplicity $1$. Additionally, we will choose the group $G$ in a way that will guarantee that the constructed $G$-invariant domains $\O_\l^{\vec{v}}$ are not merely translations of the standard annulus. 

More precisely, let $-\D_{\S^{n-1}}$ be the Laplace-Beltrami operator on the sphere $\S^{n-1}$ and let $\{ \s_i \}_{i=0}^{\infty}$ be the sequence of its eigenvalues, i.e. \mbox{$\s_i= i(i+n-2)$.} We shall fix any group $G \leqslant O(n)$ that has the following two properties:
\begin{itemize}
	\item[(P1)] If $T$ is a translation of $\R^n$ and $T(\S^{n-1})$ is a $G$-invariant set, then $T$ is trivial.

	\item[(P2)]  If $\{ \s_{i_k} \}_{k=0}^\infty$ are the eigenvalues of $-\D_{\S^{n-1}}$ when restricted to the $G$-invariant functions, then $\s_{i_k}$ has multiplicity equal to $1$, i.e. there exists a unique (up to normalization) $G$-invariant spherical harmonic of degree $i_k$, $k\in \Nz$.
\end{itemize}
Note that $i_0=0$ and that because of (P1), spherical harmonics of degree $1$ are not $G$-invariant, i.e. $i_1\geq 2$. An example of a group, satisfying both (P1) and (P2), is $G=O(n-1) \times \Z_2$. For instance, when $n=2$, $G = \Z_2 \times \Z_2$ acts on $\R^2$ by reflections with respect to the two coordinate axes; the eigenvalues of $-\D_{\S^{1}}=-\de^2_{\th}$, restricted to $\Z_2 \times \Z_2$-invariant functions, are $\s_{i_k} = (2k)^2$, $k\in \Nz$, and the corresponding eigenfunctions (up to normalization) are $\cos(2k \th)$.

Let us denote $Y_k := \mathcal{Y}_{i_k}$, where $\mathcal{Y}_{i_k}$ is the unique $G$-invariant spherical harmonic of degree $i_k\in \Nz$, normalized in the $L^2$-norm, that is,
$$ \D_{\S^{n-1}} Y_k + \s_{i_k} Y_k = 0 , \qquad  \int_{\S^{n-1}} |Y_k|^2 \;dS = 1 , \qquad k \in \Nz .$$
Finally, denote by
\begin{equation}\label{innprod}
    \langle \vec{w},\vec{z} \rangle_{\l} := \l^{n-1} \int_{\S^{n-1}} w_1 z_1 \;dS + \int_{\S^{n-1}} w_2 z_2 \;dS , \qquad \vec{w},\vec{z}\in L^2(\S^{n-1}) \times  L^2(\S^{n-1}) ,
\end{equation}
the inner product on $L^2(\S^{n-1})\times  L^2(\S^{n-1})$ induced by the standard inner product on $L^2(\de \O_\l)$ under the natural identification 
\begin{equation}\label{eq:ident}
\begin{array}{rcl}
\vec{w} &\leftrightarrow& \psi \\
(w_1(x),w_2(x)) &=& (\psi(\lambda x), \psi(x)) \quad \text{for all } x\in \S^{n-1},
\end{array}
\end{equation}
and let
$$
\|\vec{w}\|_\l := \sqrt{ \langle \vec{w},\vec{w} \rangle_{\l} }
$$
be the induced $L^2$-norm. We point out that $L_\l$ is formally self-adjoint with respect to $\langle\cdot, \cdot\rangle_\l$ (see Remark \ref{rem:SA}).

As a result of restricting to pairs of $G$-invariant functions in $C^{2,\a}(\S^{n-1})$, the linearized operator $L_\l:(C^{2,\a}_G(\S^{n-1}))^2 \to (C^{1,\a}_G(\S^{n-1}))^2$ will now possess a one-dimensional kernel at each critical value $\l_k:=\l_{i_k}^*$, $k\in \N$ --  spanned by an element of the form $\vec{z}_k=(a_kY_k, b_kY_k)$ -- and its image will be the closed subspace of co-dimension $1$ orthogonal to $\vec{z}_k$ with respect to the inner product \eqref{innprod}. Moreover, because of the strict $\l$-monotonicity of $\mu_{i_k, 1}(\l)$, the tranversality condition \[\de_\l L_{\l}|_{\l=\l_k}(\vec{z}_k) = \mu_{i_k, 1}'(\l_k) \vec{z}_k \notin \text{im } L_{\l_k}\] is going to hold. Invoking the Crandall-Rabinowitz Bifurcation Theorem \ref{crt}, we reach at the statement of the refinement to Theorem \ref{teo1}.

\begin{theorem}\label{teo2}
 Let $n\geq 2$, $\a\in (0,1)$, and let $G$ and $Y_k$, $k \in \N$, be as above. There is a strictly increasing sequence $\{ \lambda_k \}_{k=1}^\infty$ of positive real numbers with $\lim_{k\to\infty} \lambda_k = 1$ and a sequence $\{ \vec{z}_k \}_{k=1}^\infty$ of non-zero elements of the form $\vec{z}_k = (a_kY_k,b_kY_k)$ with $\|\vec{z}_k\|_{\l_k}=1$, such that for each $k \in \N$, there exists $\e > 0$ and a smooth curve 
    $$\begin{array}{ccc}
    	(-\e,\e) & \to & \left( C^{2,\a}(\S^{n-1}) \right)^2 \times (0,1) \\
        s & \mapsto & (\vec{w}(s),\l(s))
    \end{array}$$
    satisfying $\vec{w}(0) = 0$, $\l(0) = \l_k$, such that for $\vec{v}(s) \in C^{2,\a}(\S^{n-1}) \times C^{,\a}(\S^{n-1})$ defined by
    \begin{equation*}
        \vec{v}(s) = s(\vec{z}_k + \vec{w}(s)),
    \end{equation*}
    the overdetermined problem
    \begin{equation}\label{nontrivialoverdet}
        \begin{cases}
            -\D u = 1  & \text{ in } \O_{\l(s)}^{\vec{v}(s)} , \\
            u = 0 & \text{ on } \G_{1}^{\vec{v}(s)} , \\
            u = \mbox{const}>0  & \text{ on } \G_{\l(s)}^{\vec{v}(s)} ,\\
            \de_{\nu} u = const >0 & \text{ on } \de \O_\l^{\vec{v}(s)},
	    \end{cases}
	\end{equation}
	admits a positive solution $u \in C^{2,\a}(\overline{\O}_{\l(s)}^{\vec{v}(s)})$. Moreover, for every $s \in (-\e,\e)$ the two components of $\vec{w}(s) = (w_1(s),w_2(s))$ are $G$-invariant functions that satisfy
    \begin{equation}\label{ortho}
         \langle \vec{w}(s),\vec{z}_k \rangle_{\l_k} = 0.
    \end{equation}
\end{theorem}

Let us show how Theorem \ref{teo2} entails Theorem \ref{teo1}.
\begin{proof}[Proof of Theorem \ref{teo1}]
Fix any $k\in \N$ and $\a\in (0,1)$. We need only explain why for $s\neq 0$ the $C^{2,\a}$ domains $\O_{\l(s)}^{\vec{v}(s)}$, constructed in Theorem \ref{teo2}, are different from a standard annulus, and why they are actually smooth. 

Since the functions $v_1(s)$, $v_2(s)$, are $G$-invariant, the corresponding domains $\O_{\l(s)}^{\vec{v}(s)}$ and solutions $u$ of \eqref{nontrivialoverdet} are also $G$-invariant. We point out that the orthogonality condition \eqref{ortho} implies that for $s\neq 0$, the non-zero $\vec{v}(s)$ is also non-constant, which means that at least one of the boundary components $\G_{r}^{\vec{v}(s)}$, $r=\l(s),1$, is different from a central dilation of $\S^{n-1}$ with respect to the origin. In addition, Property (P1) of $G$ prevents $\O_{\l(s)}^{\vec{v}(s)}$ from being an affine transformation of the annulus $\O_{\l}$ that involves a non-trivial translation. All this guarantees the nontriviality of $\O_{\l(s)}^{\vec{v}(s)}$. 

The domains $\O=\O_{\l(s)}^{\vec{v}(s)}$ are constructed to be of class $C^{2,\a}$, but by the classical regularity result of Kinderlehrer and Nirenberg \cite{KindNiren}, the boundary of a $C^{2,\a}$-domain $\O$, supporting a solution $u\in C^{2,\a}(\overline{\O})$ to \eqref{eq:prob1}, gets upgraded to a smooth one (in fact, to an analytic one). The solution $u$ itself is in $C^{\infty}(\overline{\O})$. 
\end{proof}

\section{Reformulating the problem and deriving its linearization}\label{probset}

Let us first recast the operator $F_\l$, defined in \eqref{eq:defFlambda}, by pulling back the Dirichlet problem \eqref{eq:prob2} from $\O_\l^{\vec{v}}$ to the annulus $\O_\l$, where we shall use polar coordinates
\[
(0,\infty) \times \S^{n-1} \cong \R^n \setminus \{0\} \quad \text{under} \quad (r,\th) \mapsto x = r\th
\] 
to describe the geometry. In this way, $\O_\l \cong (\l,1) \times \S^{n-1}$, its boundary components $\G_\l \cong \{ \l \} \times \S^{n-1}$, $\G_1 \cong \{ 1 \} \times \S^{n-1}$ are two copies of $\S^{n-1}$, and we  naturally get the identification of functions \eqref{eq:ident}. 

For any $\vec{v} = (v_1,v_2) \in U \subseteq (C^{2,\a}(\S^{n-1}))^2$ of sufficiently small norm, we consider the diffeomorphism $\F: \O_\l \to \O_{\l}^{\vec{v}}$ defined in polar coordinates by
	\begin{equation}\label{eq:diffeo}
    	\F(r,\th) = \left(\big(1 + \eta_1(r) v_1(\th) + \eta_2(r) v_2(\th)\big)r, \th \right) ,
    \end{equation}
where $\eta_j$, $j=1,2$, are smooth functions satisfying
	\begin{equation}\label{cutoff}
     	\eta_1(r) = \begin{cases} 1/\l & \text{ if } r \leq \l + \d \\ 0 & \text{ if } r \geq \l +2\d \end{cases} , \qquad
    	\eta_2(r) = \begin{cases} -1 & \text{ if } r \geq 1-\d \\   0 & \text{ if } r \leq 1 -2\d \end{cases} ,
	\end{equation}
for some small enough $\d>0$.  We set on $\O_\l$ the pull-back metric $g = \F^*g_0$, where $g_0$ is the Euclidean metric on $\O_{\l}^{\vec{v}}$. Near $\de \O_\l$, the metric $g$ equals
	\begin{equation}\label{metricgv}
		g = (1 + \eta_j v_j)^2dr^2 + 2r \eta_j \; (1 + \eta_j v_j)drdv_j + r^2\eta_j^2 dv_j^2 + r^2(1 + \eta_j v_j)^2 g_{\S^{n-1}} ,
	\end{equation}
where $g_{\S^{n-1}}$ is the standard metric on $\S^{n-1}$. Since $\F$ is an isometry between $(\O_{\l}, g)$ and $(\O_\l^\vec{v}, g_0)$, $u_\l(\vec{v})$ is the solution of the Dirichlet problem \eqref{eq:prob2} in $\O_{\l}^{\vec{v}}$ if and only if $u_\l^*(\vec{v}) := \F^*u_\l(\vec{v})$ is the solution of
	\begin{equation}\label{eq:pullbackprob}
    	\begin{cases}
        	-\D_{g} u = 1 & \text{ in } \O_\l , \\
            u = 0 & \text{ on } \G_1 , \\
            u = a_\l & \text{ on } \G_\l .
        \end{cases}
    \end{equation}
By Schauder theory, $u_\l^*(\vec{v}) \in C^{2,\a}(\overline{\O}_\l)$ and it depends smoothly on $\vec{v}$.

Let $\n^*$ denote the \emph{inner} unit normal field to $\p\O_\l$ with respect to the metric $g$. We have $\F_*\n^* = \n$, and to find the expression for $\p_\n u_\l(\vec{v})$ in these coordinates, we need to compute $\p_{\n^*}u_\l^*(\vec{v}) = g(\nabla_g u_\l^*(\vec{v}), \n^*)$. In the new coordinates, the operator $F_\l: U \to (C^{2,\a}(\S^{n-1})^2$ thus becomes 
	\begin{equation}\label{eq:Flambdapolar}
		F_\l(\vec{v}) := \frac{1}{c_\l}\left( \left.\frac{\p u_\l^*(\vec{v})}{\p\n^*} \right|_{\G_\l} - \left.\frac{\p u_\l}{\p\n} \right|_{\G_\l} , \left.\frac{\p u_\l^*(\vec{v})}{\p\n^*} \right|_{\G_1} -\left.\frac{\p u_\l}{\p\n} \right|_{\G_1} \right).
	\end{equation}
Then $F_\l(\vec{0}) = \vec{0}$ for all $\l \in (0,1)$, since $u_\l^*(\vec{0}) = u_\l$, and $F_\l(\vec{v}) = \vec{0}$ if and only if $\p_{\n^*} u_\l^*(\vec{v})=c_\l$ is constant on $\p\O_\l$. The latter is equivalent to $u_\l(\vec{v})$ solving the overdetermined problem \eqref{eq:prob1}-\eqref{BC} in $\O_\l^{\vec{v}}$.

In the following proposition we will compute the linearization at $\vec{v}=\vec{0}$ of the operator $F_\l(\vec{v})$, reformulated as in \eqref{eq:Flambdapolar}. Recall that we use the identification \eqref{eq:ident} of functions on $\de \O_\l$ with a pair of functions on $\S^{n-1}$. 

\begin{prop}\label{prop:linearization}
    The smooth operator $F_\l$, defined in \eqref{eq:Flambdapolar}, has a linearization at $\vec{v} = \vec{0}$, \\ $L_\lambda := D_\vec{v} F_\lambda|_{\vec{v}=\vec{0}} : (C^{2,\a}(\S^{n-1}))^2 \to (C^{1,\a}(\S^{n-1}))^2$ given by 
    \begin{equation}\label{eq:defLlambda}
        L_\lambda(w_1,w_2) = \left(  \left.\frac{\p\f_\vec{w}}{\p\n} \right|_{\G_{\l}} + \frac{w_1}{c_\l} \left.\frac{\p^2 u_\l}{\p r^2} \right|_{\G_{\l}}, 
        \left.\frac{\p\f_\vec{w}}{\p\n} \right|_{\G_{1}} + \frac{w_2}{c_\l} \left.\frac{\p^2 u_\l}{\p r^2} \right|_{\G_{1}}
        \right) ,
    \end{equation}
    where for $\vec{w}=(w_1, w_2) \in (C^{2,\a}(\S^{n-1}))^2$, $\f_\vec{w}$ denotes the harmonic function $\f\in C^{2,\a}(\overline{\O}_\l)$ with boundary values 
    $\f|_{\G_\l}(x) = w_1(x/|x|)$ and $\f|_{\G_1}(x) = w_2(x/|x|)$.
    \begin{proof}
        As $F_\l$ is a smooth operator, its linearization at $\vec{0}$ is given by the directional derivative
        $$ L_\lambda(\vec{w}) = \lim_{t\to0} \frac{F_\lambda(t\vec{w})}{t} .$$
        Write $\vec{v}=(v_1, v_2)=t(w_1, w_2)$ for small $t$ and consider the diffeomorphism $\F = \F_t$ defined in \eqref{eq:diffeo} and the induced metric $g = g_t$ on $\O_\l$. Let $u_\l^*(\vec{v}) = u_t$ be the solution of the Dirichlet problem \eqref{eq:pullbackprob} in $\O_\l$, which smoothly depends on the parameter $t$. Since $u_\l$ is a radial function and can be extended by \eqref{eq:trivialsol} to solve $-\D_{g_0} u_\l =1$ in the whole of $\R^n \setminus \{ 0 \}$, we have that $u_\l^* := \F_t^*u_\l$ is well defined and solves
        $$ -\D_{g_t} u = 1 \quad \text{ in } \quad \O_\l. $$         
        Expanding  $u_\l^* = u_\l^*(r,\th)$ in a neighbourhood of $\p\O_\l$ to first order in $t$, we obtain
        $$u_\l^*(r,\th) = u_\l((1+t\eta_jw_j)r,\th) = u_\l(r) + tr \eta_j(r) w_j(\th) \frac{\p u_\l}{\p r} + O(t^2), $$
        where $\eta_j$, $j=1,2$, are the functions defined in \eqref{cutoff}.
        
        Let $\psi_t := u_t - u_\lambda^*$. Then $\psi_t \in C^{2,\a}(\overline{\O}_\l)$ is a solution of
        \begin{equation}\label{eq:eqpsi}
            \begin{cases}
            \D_{g_t} \psi_t = 0  & \text{ in } \O_\l , \\ 
            \psi_t = -u_\lambda^*  & \text{ on } \G_1 , \\
            \psi_t = a_\lambda - u_\lambda^* & \text{ on } \G_\l , \\ 
            \end{cases}
        \end{equation}
        which depends smoothly on $t$, with $\psi_0 = 0$. Setting $\dot{\psi} := \frac{d}{dt}\psi_t \big|_{t=0}$, we can differentiate \eqref{eq:eqpsi} at $t=0$ to obtain
        \begin{equation*}
            \begin{cases}
            \D \dot{\psi} = 0  & \text{ in } \O_\l , \\ 
            \dot{\psi} = (\p_ru_\lambda) w_2  = -c_\l w_2& \text{ on } \G_1 , \\
            \dot{\psi} = -(\p_ru_\lambda) w_1 = -c_\l w_1& \text{ on } \G_\l , \\ 
            \end{cases}
        \end{equation*}
        so that 
        \begin{equation}\label{eq:psiphi}
       \dot{\psi} =  -c_{\l} \phi_{\vec{w}}.
       \end{equation}
        Now, given that $\psi_t = t \dot{\psi} + O(t^2)$, we have in a neighbourhood of $\p\O_\lambda$
        \begin{equation}\label{approxu}
            u_t = u_\lambda + t \left( \dot{\psi} + r\eta_j w_j \frac{\p u_\l}{\p r} \right) + O(t^2).
        \end{equation}
        
        Recall that $\n^* = \n_t$ denotes the \emph{inner} unit normal field to $\p\O_\l$ with respect to the metric $g_t$. In the following we compute $\p_{\n_t}u_t$ to first order in $t$. As $u_t$ is constant on each boundary component, it follows that $\p_{\n_t} u_t = |\nabla_{g_t} u_t|$ on $\p\O_\l$. Using formula \eqref{metricgv} for $g_t$ in a neighbourhood of $\p\O_\l$, we find
        $$g_t^{-1} = \left( 1 - 2t\eta_j w_j + O(t^2) \right) dr^2 + O(t) drdw_j + \left( \frac{1}{r^2} + O(t) \right) d\th^2.$$
        Now, $|\nabla_{g_t} u_t|^2 = g_t^{rr}(\p_r u_t)^2$ on $\p\O_\l$ and after differentiating \eqref{approxu} with respect to $r$, we obtain
        \begin{equation}\label{eq:linearization1}
            |\nabla_{g_t} u_t| = \frac{\p u_\lambda}{\p\n} + t \left( \frac{\p \dot{\psi}}{\p\n} + w_j \frac{\p^2 u_\l}{\p r^2} \right) + O(t^2) \quad \text{on}\quad \de \O_\l.
        \end{equation}
        The formula \eqref{eq:defLlambda} for $L_\l$ hence follows from \eqref{eq:linearization1} and \eqref{eq:psiphi}.
        \end{proof}
    
\end{prop}

\section{Spectrum of the linearized operator}\label{opspec}

In this section we give an account of the spectral properties of the linearized operator $L_\l$, which we derived in Proposition \ref{prop:linearization}.

Recall that a function $\mathcal{Y} \in C^{\infty}(\S^{n-1})$ is a spherical harmonic of degree $k \in \Nz$ if it is an eigenfunction of the Laplace-Beltrami operator $-\D_{\S^{n-1}}$ on $\S^{n-1}$, that is,
$$ \D_{\S^{n-1}} \mathcal{Y} + \s_k \mathcal{Y} = 0 ,$$
where $\s_k := k(k+n-2)$ is the corresponding eigenvalue. We first observe that the subspace $W$ generated by $\{ (\mathcal{Y},0), (0,\mathcal{Y}) \}$ is invariant under $L_\l$ and we shall derive a matrix representation of $L_\l|_{W}$ with respect to a certain convenient basis.

\begin{lemma}\label{invariance}
    Let $\mathcal{Y} \in C^{\infty}(\S^{n-1})$ be a spherical harmonic of degree $k\in \Nz$ and unit $L^2(\S^{n-1})$ norm, and let $W$ be the subspace of $C^{\infty}(\S^{n-1}) \times C^{\infty}(\S^{n-1})$ spanned by $\{ (\mathcal{Y},0), (0,\mathcal{Y}) \}$. 
Then $W$ is invariant under the action of $L_\l$. Moreover, if 
\begin{equation}\label{eq:basis}
\vec{e}_1 := (\l^{\frac{1-n}{2}} \mathcal{Y},0), \quad \vec{e}_2 := (0,\mathcal{Y}),
\end{equation}     
then $\mathcal{B}=\{\vec{e}_1, \vec{e}_2\}$ is an orthonormal basis for $W$ with respect to the inner product $\langle \cdot, \cdot \rangle_\l$ on $L^2(\S^{n-1})\times L^2(\S^{n-1})$, defined in \eqref{innprod}, and the matrix representing the restriction $L_\l|_{W}$ with respect to the basis $\mathcal{B}$ is given by
\begin{equation}\label{matrix}
    M_{\l,k} = \left( \begin{array}{cc}
        \frac{1}{\l} \frac{(k+n-2)\l^{2-n-k} + k\l^k}{\l^{2-n-k} - \l^k} - \frac{n-1}{\l} & \l^{\frac{1-n}{2}} \frac{2-n-2k}{\l^{2-n-k} - \l^k}  \\
        \l^{\frac{1-n}{2}} \frac{2-n-2k}{\l^{2-n-k} - \l^k} & \frac{k\l^{2-n-k} + (k+n-2)\l^k}{\l^{2-n-k} - \l^k} + (n-1)
    \end{array} \right) - \frac{1}{c_\l} \id, \qquad k\geq 1,
\end{equation}
while for $k=0$, 
\begin{align}
 M_{\l,0} &= \left( \begin{array}{cc}   \frac{1}{\l} \frac{(n-2)\l^{2-n}}{\l^{2-n} - 1} - \frac{n-1}{\l} & \l^{\frac{1-n}{2}} \frac{2-n}{\l^{2-n} - 1}  \\
        \l^{\frac{1-n}{2}} \frac{2-n}{\l^{2-n} - 1} & \frac{n-2}{\l^{2-n} - 1} + (n-1)\end{array} \right) - \frac{1}{c_\l} \id \quad \text{when} \quad n\geq 3, \label{matrix03} \\
       M_{\l,0} &= \left( \begin{array}{cc}   -\frac{1}{\l} \frac{1}{\log\l} - \frac{1}{\l} &  \l^{-\frac{1}{2}} \frac{1}{\log\l}  \\
        \l^{-\frac{1}{2}} \frac{1}{\log\l} & - \frac{1}{\log\l} + 1\end{array} \right) - \frac{1}{c_\l} \id  \quad\text{when} \quad n=2.  \label{matrix02}
\end{align}

    \begin{proof}
        It is easy to verify that $\mathcal{B}$ is an orthonormal basis for $W$ with respect to the inner product $\langle \cdot, \cdot \rangle_\l$. Let $\vec{w} = a \vec{e}_1 + b \vec{e}_2$. For $n\geq 3$ and $k\in \Nz$, as well as for $n=2$, $k\in \N$,  the function $\f_{\vec{w}} \in C^\infty(\overline{\O}_\l)$, defined in polar coordinates $r \in [\l,1]$, $\th \in \S^{n-1}$ by 
        $$\f_{\vec{w}}(r,\th) = (aA(r) + bB(r)) \mathcal{Y}(\th) ,$$ 
        where
        \begin{equation}\label{eq:AB}
            A(r) = \l^{\frac{1-n}{2}} \frac{r^{2-n-k}-r^k}{\l^{2-n-k}-\l^k} , \qquad B(r) =  \frac{\l^{2-n-k}r^k-\l^k r^{2-n-k}}{\l^{2-n-k}-\l^k} ,
        \end{equation}
        is harmonic in $\O_\l$ and satisfies $$\f_{\vec{w}}(\l,\th) = a  \l^{\frac{1-n}{2}} \mathcal{Y}(\th), \quad \f_{\vec{w}}(1,\th) = b \mathcal{Y} (\th).$$
        Formula \eqref{eq:defLlambda} for $L_{\l}$ hence gives
        \begin{align}\label{eq:invariance}
             & L_\l(a \vec{e}_1 + b \vec{e}_2) = \left( -(aA'(\l)+bB'(\l)) \mathcal{Y} +\frac{\p_r^2 u_\l(\l)}{c_\l} a\l^{\frac{1-n}{2}} \mathcal{Y} , (aA'(1)+bB'(1)) \mathcal{Y} + \frac{\p_r^2 u_\l(1)}{c_\l} b\mathcal{Y} \right) \notag \\
            &=  \left( -(aA'(\l)+bB'(\l))\mathcal{Y} - \frac{n-1}{\l} a\l^{\frac{1-n}{2}} \mathcal{Y} , (aA'(1)+bB'(1)) \mathcal{Y} + (n-1) b\mathcal{Y} \right) - \frac{1}{c_\l}(a \vec{e}_1 + b \vec{e}_2) ,
        \end{align}
where in the last equality we used the fact that $\de_r^2 u_\l = -1 - \frac{n-1}{r} \de_r u_\l$. Therefore, $W$ is an invariant subspace of $L_\l$. Plugging in \eqref{eq:AB} in \eqref{eq:invariance}, we easily derive formulas \eqref{matrix} and \eqref{matrix03} for the matrix representation $M_{\l,k}$ of $L_\l|_{W}$ with respect to the basis $\mathcal{B}$.

When $n=2$, and $k=0$, the substitute for \eqref{eq:AB} is 
\begin{equation}\label{eq:AB2}
A(r) = \l^{-1/2} \frac{\log r}{\log \l}, \qquad B(r) = -\frac{\log r - \log \l}{\log \l}
\end{equation}
and we easily check again the invariance of $W$ under $L_\l$ and derive \eqref{matrix02}.
    \end{proof}
\end{lemma}

\begin{remark}\label{rem:SA}  
Note that the matrix $M_{\l,k}$ in Lemma \ref{invariance} is symmetric. This is not surprising taking into account the fact that the operator $L_{\l}$ is formally self-adjoint with respect to the inner product $\langle \cdot,\cdot \rangle_\l$ over the space $V:=(C^{2,\a}(\S^{n-1}))^2$. Indeed, for $\vec{w_1}, \vec{w_2} \in V$, identified as functions on $\de \O_{\l}$ under \eqref{eq:ident}, let $\phi_{\vec{w_1}}, \phi_{\vec{w_2}}$ denote their corresponding harmonic extensions to $\overline{\O}_\l$. Using formula \eqref{eq:defLlambda} for $L_\l$ and the definition \eqref{innprod} of the inner product $\langle \cdot,\cdot \rangle_\l$, an application of Green's formula and harmonicity yield
\begin{align*}
\langle L_\l \vec{w_1}, \vec{w_2} \rangle_\l & = \int_{\de \O_\l} \left(-\frac{\de\phi_{\vec{w_1}}}{\de\nu} + \frac{\de^2 u_\l}{\de r^2} \frac{\phi_{\vec{w_1}}}{c_\l} \right)\phi_{\vec{w_2}} \; dS \\
& = \int_{\de \O_\l} \phi_{\vec{w_1}} \left(-\frac{\de\phi_{\vec{w_2}}}{\de\nu} + \frac{\de^2 u_\l}{\de r^2} \frac{\phi_{\vec{w_2}}}{c_\l} \right) \; dS = \langle \vec{w_1}, L_\l \vec{w_2} \rangle_\l.
\end{align*}
\end{remark}

\begin{remark}\label{rem:analytic}
Note that the matrices $M_{\l, k}$, $\l \in (0,1)$, $k\in \Nz$, derived in Lemma \ref{invariance}, fit in a two-parameter family of symmetric matrices 
\begin{equation}\label{eq:matrixfam}
\mathcal{M}=\{M_{\l,k}: (\l, k)\in (0,1)\times [0,\infty)\}.
\end{equation}
where we define $M_{\l, k}$ for non-integral $k\in [0,\infty)$ by the analytic formula in equation \eqref{matrix}. In that way, the family $\mathcal{M}$ is analytic in both $(\l, k)\in (0,1)\times (0,\infty)$. Moreover, we can see that $\mathcal{M}$ is continuous up to $(\l, k)\in (0,1)\times \{0\}$ after easily checking
\[
\lim_{k\to 0^+} M_{\l,k} = M_{\l,0},
\]
where $M_{\l, 0}$ is given by \eqref{matrix02} when $n=2$ and by \eqref{matrix03} when $n=3$. 

The symmetric matrices of $\mathcal{M}$ are never multiples of the identity (the off-diagonal entries are non-zero), whence every $M_{\l,k} \in \mathcal{M}$ has two distinct real eigenvalues $$\m_{k,1}(\l)<\m_{k,2}(\l)$$ and each is a smooth function of  $(\l, k) \in (0,1)\times (0,\infty)$, continuous up to $(\l, k)\in (0,1)\times \{0\}$. Furthermore, since any eigenvector of $M_{\l,k}\in \mathcal{M}$ has two non-zero entries, we can define $v_j(\l,k)\in \R^2$ to be the unique eigenvector of $M_{\l,k}$, associated with $\m_{k,j}(\l)$, $j=1,2$, that has unit Euclidean norm and positive first entry. Clearly, the eigenvector $v_j(\l,k)$, $j=1,2$, depends smoothly in $(\l, k) \in (0,1)\times (0,\infty)$ and is contininuous up to $(\l, k)\in (0,1)\times \{0\}$, as well.

\end{remark}

Before we continue, it will be convenient to recast the matrices $M_{\l,k}\in \mathcal{M}$ in new notation that will greatly facilitate the computations when we analyze the behaviour of its eigenvalues $\m_{k,j}(\l)$, $j=1,2$. For that purpose, first define the matrix
\begin{equation}\label{modmatrix}
	\tilde{M}_{\l,k} := M_{\l, k} + \frac{1}{c_\l} \id , \qquad M_{\l,k}\in \mathcal{M}
\end{equation}
whose eigenspaces are the same as those of $M_{\l,k}$ and whose eigenvalues $\tilde{\m}_{k,j}(\l)$ are shifts of $\m_{k,j}(\l)$ by $1/c_\l$:
\begin{equation}\label{eq:eigvalshift}
        \m_{k,j}(\l) = \tilde{\m}_{k,j}(\l) - \frac{1}{c_\l} , \qquad j=1,2 .
    \end{equation}
For a given $k \in (0,\infty)$ we shall denote
\begin{equation}\label{eq:alphaomega}
\a := \frac{n}{2}+k-1, \quad \a \in (0,\infty)\quad \text{and}\quad e^\o:=\l^{-\a}, \quad \o \in (0,\infty).
\end{equation}

\begin{lemma}
	Let $k \in (0,\infty)$ be given and let $\a,\o$ be as in \eqref{eq:alphaomega}. The matrix $\tilde{M}_{\l,k}$ defined in \eqref{modmatrix} takes the form
	\begin{equation}\label{modmatrix2}
           \tilde{M}_{\l,k} = \left( \begin{array}{cc}
                \frac{1}{\l}(\a \coth \o - \frac{n}{2}) & -\frac{\a}{\sqrt{\l}}\frac{1}{\sinh \o} \\
                -\frac{\a}{\sqrt{\l}}\frac{1}{\sinh \o} & \a \coth \o + \frac{n}{2}
            \end{array} \right) 
        \end{equation}
	and its eigenvalues are given by
	\begin{equation}\label{modeigenvals}
            \tilde{\m}_{k,j}(\l) = \frac{C \mp \sqrt{C^2 - 4\l D}}{2\l} , \qquad j=1,2,
	\end{equation}
	where
	\begin{equation}\label{eq:CD}
		C = \a(\l+1)\coth \o + \frac{n}{2}(\l-1) , \qquad  D = \a^2 - \frac{n^2}{4} = (n+k-1)(k-1) .
	\end{equation}
	
	\begin{proof}
		These are straightforward computations. First note that
		\begin{align*}
			(k+n-2)\l^{1-n/2-k} + k\l^{n/2+k-1} & = (2\a-k)\l^{-\a} + k\l^\a \\
			& = \left( \frac{n}{2}-1 \right)(\l^{-\a} - \l^\a) + \a(\l^{-\a} + \l^\a)
		\end{align*}
		and, using the expresion in \eqref{matrix}, we find the $(1,1)$-entry of $\tilde{M}_{\l,k}$ to be
		\begin{align*}
			\frac{1}{\l}\left(\frac{(k+n-2)\l^{1-n/2-k} + k\l^{n/2+k-1}}{\l^{1-n/2-k} - \l^{n/2+k-1}} - (n-1) \right)& =\frac{1}{\l}\left( \left( \frac{n}{2}-1 \right) + \a\frac{\l^{-\a} + \l^\a}{\l^{-\a} - \l^\a} -(n-1) \right) \\
			& = \frac{1}{\l}\left(\a\coth\o - \frac{n}{2}\right).
		\end{align*}
		In a similar fashion we compute the $(2,2)$-entry of $\tilde{M}_{\l,k}$. Also
		$$ \l^{\frac{1-n}{2}} \frac{2-n-2k}{\l^{2-n-k}-\l^k} = -\frac{2\a}{\sqrt{\l}} \frac{1}{\l^{-\a}-\l^\a} = \frac{-\a}{\sqrt{\l}}\frac{1}{\sinh\o} .$$
		This establishes equation \eqref{modmatrix2}. The characteristic equation for $\tilde{M}_{\l,k}$ then computes to
		$$ \l\tilde{\m}^2 - \left\{ \a(\l+1)\coth\o +\frac{n}{2}(\l-1) \right\} \tilde{\m} + \left\{ \a^2 - \frac{n^2}{4} \right\} = 0 ,$$
		from which we derive formulas \eqref{modeigenvals}-\eqref{eq:CD} for its eigenvalues.
	\end{proof}
\end{lemma}

\begin{remark}\label{rem:k1evals}
	Note that when $k=1$, we have $\a = n/2$ and $e^\o = \l^{-n/2}$, so that $C$ and $D$ in \eqref{eq:CD} evaluate to
	\[
	C=\frac{n}{2}\left((\l + 1)\frac{\l^{-n/2} + \l^{n/2}}{\l^{-n/2} - \l^{n/2}} + \l -1 \right) = n \frac{\l + \l^{n}}{1-\l^n} = \frac{\l}{c_\l}, \qquad D=0.
	\]
	Hence, \eqref{modeigenvals} gives us $\tilde{\m}_{1,1}(\l) = 0$ and $\tilde{\m}_{1,2}(\l)=1/c_\l$, and the eigenvalues of the matrix $M_{\l,1}$  are then given by 
	\begin{equation}\label{eq:k1evals}
	\m_{1,1}(\l) = -\frac{1}{c_\l} , \qquad \m_{1,2}(\l) = 0.
	\end{equation} 
	We observe that, over subspaces $W=\textrm{Span}\{ (\mathcal{Y},0), (0,\mathcal{Y}) \}$, where $\mathcal{Y}$ is a spherical harmonic of degree $1$, the linearization $L_{\l}|_W$ has a kernel of dimension $1$ for every $\l \in (0,1)$. This kernel precisely corresponds to deformations of the standard annulus $\O_\l$ generated by translations. 
	
\end{remark}

In the following key sequence of lemmas we will examine the behaviour of the eigenvalues $\mu_{k,1}(\l)$ and $\mu_{k,2}(\l)$ in both $k\in \Nz$ and $\l\in (0,1)$. See Figure \ref{fig:evals} below for a plot of these branches for $k=0,1,2,3$, in dimension $n=3$. 
\begin{figure}[hbtp]
\centering \includegraphics[scale=0.6]{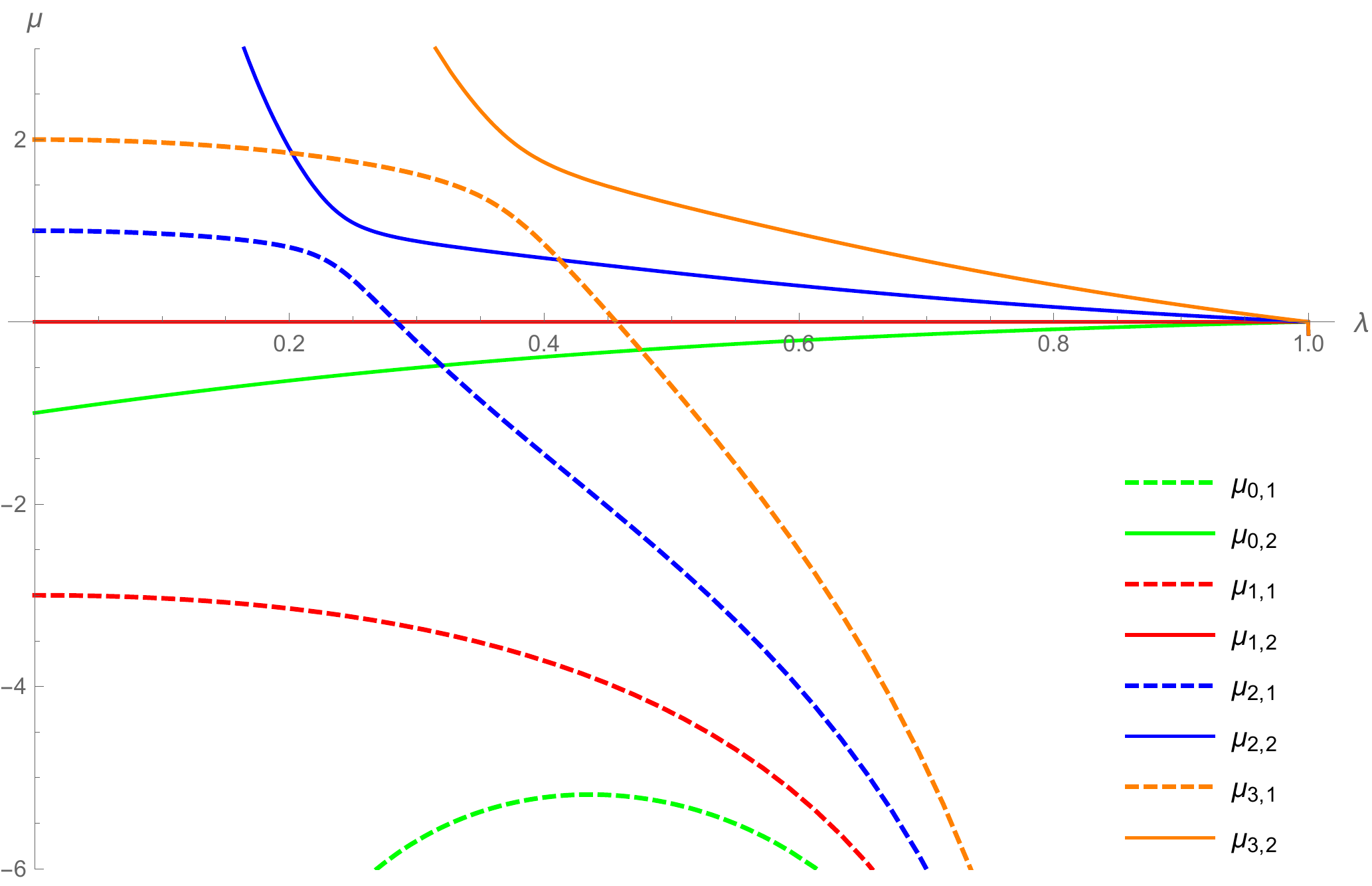} \caption{Mathematica plot of the eigenvalues $\mu_{k,j}$, $k=0,1,2,3$ and $j=1,2$, as a function of $\l\in (0,1)$ for $n=3$.}
\label{fig:evals}
\end{figure}


First, we will establish the first branch $\mu_{k,1}(\l)$ is strictly decreasing in $\l\in (0,1)$ for any given $k \geq 2$. The proof of the next lemma is based on a delicate use of hyperbolic trigonometric identities. 
\begin{lemma}\label{eigenvalsdecrease}
	For $k\in \N$, let $M_{\l,k}$ be the matrix defined in \eqref{matrix} and let $\m_{k,1}: (0,1) \to \R$ denote its first eigenvalue. For every $k \geq 2$ the following are satisfied:
	\begin{itemize}
		\item $\disp \lim_{\l\to0} \m_{k,1}(\l) = k-1$, $\disp \lim_{\l\to1} \m_{k,1}(\l) = -\infty$;
		\item $\m_{k,1}'(\l) < 0$, and so $\m_{k,1}(\l)$ is strictly decreasing in $\l$.
	\end{itemize}
	\begin{proof}
		Fix $k \in \R$, $k>1$. We shall first prove that $\m'_{k,1}(\l) < 0$, where $'$ denotes differentiating with respect to $\l$. As $\frac{1}{c_\l}=n\frac{1+\l^{n-1}}{1-\l^{n}} $ is strictly increasing in $\l$,  by \eqref{eq:eigvalshift} it suffices to show that $\tilde{\m}'_{k,1}(\l) \leq 0$. 
        Since by \eqref{modeigenvals}
        \[ \tilde{\m}_{k,1}(\l) = \frac{2D}{C + \sqrt{C^2-4\l D}} \]
        and given that $D=(n+k-1)(k-1)$ is positive and constant in $\l$, we only need to show that
        \begin{equation}\label{eq:derivinl}
        \frac{\p}{\p\l} \left( C + \sqrt{C^2-4\l D} \right) =  \frac{C'(\sqrt{C^2-4\l D}+C)-2D}{\sqrt{C^2-4\l D}} > 0.
        \end{equation}
        Using the identity $\coth^2\o - 1/\sinh^2\o = 1$ successively, we get 
        \begin{align*}
            C^2 - 4\l D & = \left\{ \a(1+\l)\coth \o - \frac{n}{2}(1-\l) \right\}^2 - 4\l \left\{ \a^2 - \frac{n^2}{4} \right\} \\
                        & = \a^2(1+\l)^2 \left( 1 + \frac{1}{\sinh^2 \o} \right) - n\a(1+\l)(1-\l)\coth \o + \frac{n^2}{4}(1-\l)^2 - 4\l \left( \a^2 - \frac{n^2}{4} \right) \\
                        & = \a^2(1-\l)^2 - n(1+\l)(1-\l)\coth\o + \frac{n^2}{4}(1+\l)^2 + \frac{\a^2(1+\l)^2}{\sinh^2\o} \\
                        & = \a^2(1-\l)^2 \left( 1 + \frac{1}{\sinh^2\o} \right) - n(1+\l)(1-\l)\coth\o + \frac{n^2}{4}(1+\l)^2 + \frac{4\l\a^2}{\sinh^2\o} \\
                        & = \left\{ \a(1-\l) \coth \o - \frac{n}{2}(1+\l) \right\}^2 + \frac{4\l\a^2}{\sinh^2 \o} ,
        \end{align*}        
        which gives the estimate 
        \begin{equation}\label{eq:monoest1}
        \sqrt{C^2-4\l D} > \a(1-\l)\coth \o - \frac{n}{2}(1+\l).
        \end{equation}
         On the other hand, using the fact that $\o' = -\a/\l$,
         \begin{equation}\label{eq:monoest2}
        C' = \a\coth \o + \frac{1+\l}{\l} \frac{\a^2}{\sinh^2 \o} + \frac{n}{2} > \a\coth \o + \frac{\a^2}{\sinh^2 \o} + \frac{n}{2},
        \end{equation}
        so that \eqref{eq:monoest1} and \eqref{eq:monoest2} yield
        \begin{align*}
            C'(\sqrt{C^2-4\l D} + C) - 2D & > \left( \a\coth \o + \frac{\a^2}{\sinh^2 \o} + \frac{n}{2} \right) \left( 2\a\coth \o - n \right) - \left( 2\a^2 -\frac{n^2}{2} \right) \\
            & = 2\a^2 \left( 1 + \frac{1}{\sinh^2\o} \right) + \frac{2\a^3\coth\o}{\sinh^2\o} - \frac{n\a^2}{\sinh^2\o} - 2\a^2 \\
            & = \frac{2\a^2}{\sinh^2 \o} \left( 1 + \a\coth \o - \frac{n}{2} \right) \\
            & > \frac{2\a^2}{\sinh^2 \o} \left( 1+\a-\frac{n}{2} \right) = \frac{2\a^2k}{\sinh^2 \o} > 0,
        \end{align*}
        where in the penultimate inequality we used $\coth \o >1$. This confirms \eqref{eq:derivinl} and completes the proof of the strict monotonicity of $\m_{k,1}(\l)$ in $\l$.
         
         To derive the limiting behaviour of $\m_{k,1}(\l)$ as $\l \to 0$, we first note that from the definition \eqref{eq:alphaomega}, we have $\lim_{\l\to0} \o = \infty$, so that \eqref{eq:CD} gives $\lim_{\l\to0} C = \a - n/2 = k-1$, and since $\lim_{\l\to0} 1/c_\l = n$, we calculate
         $$ \lim_{\l\to0} \m_{k,1}(\l) = \lim_{\l\to0} \left(\frac{2D}{C + \sqrt{C^2-4\l D}} - \frac{1}{c_\l}\right) =  \frac{D}{k-1} - n = k-1 .$$
         As to the limiting behaviour of  $\m_{k,1}(\l)$ as $\l \to 1$, the fact that $\tilde{\m}_{k,1}(\l)$ is decreasing in $\l$ and $\disp \lim_{\l\to1} 1/c_\l = \infty$  yields $$\lim_{\l\to 1} \m_{k,1}(\l) = -\infty.$$
	\end{proof}
\end{lemma}

Next, we will prove that, for fixed $\l$, both $\m_{k,1}(\l)$ and $\m_{k,2}(\l)$ increase with $k$. As the explicit formulas \eqref{modeigenvals}-\eqref{eq:CD} for the eigenvalues turn out to be unyielding, we accomplish this instead by treating $k$ as a continuous variable and showing that $\p_k M_{\l,k}$ is positive definite. 
\begin{lemma}\label{eigenvalsincrease}
	For fixed $\l \in (0,1)$ and $j=1,2$ the sequence $\{ \m_{k,j}(\l) \}_{k=0}^{\infty}$ is strictly increasing. 
	\begin{proof}
		We shall treat $k\in [0,\infty)$ as a continuous variable, following the discussion in Remark \ref{rem:analytic}. First, we restrict ourselves to $k>0$ and fix $\l$. Recall that in the remark we defined $v = v_j(\l,k)\in \R^2$ to be the unique eigenvector of $M_{\l,k}$ associated with eigenvalue $\m_{k,j}(\l)$, $j=1,2$, which has unit Euclidean norm and positive first entry. Then $\p_k v \in \R^2$ is orthogonal to $v$ and since $M_{\l,k}$ is symmetric, we have
        \begin{equation}\label{eq:orthok}
        \langle M_{\l,k}(\p_k v),v \rangle = \langle \p_k v, M_{\l,k} v \rangle =\m_{k,j} \langle \p_k v, v \rangle =0.
        \end{equation}
         Differentiating the identity $\m_{k,j}=\langle M_{\l,k}v,v \rangle$ with respect to $k$ and using \eqref{eq:orthok}, we obtain
        $$ \p_k \m_{k,j} = \langle (\p_k M_{\l,k})v,v \rangle + \langle M_{\l,k}(\p_k v),v \rangle + \langle M_{\l,k} v,\p_k v \rangle = \langle \p_k M_{\l,k}v,v \rangle.$$
        Therefore, we will have the desired $\p_k \m_{k,j}>0$ once we show that the symmetric matrix $\p_k M_{\l,k}$ is positive definite. Using $\de \a/\de k = 1$ and $\de \o/\de k =\o /\a$, we compute from \eqref{modmatrix2}
        \begin{equation*}
            \p_k M_{\l,k} = \p_k \tilde{M}_{\l,k} = \left(\begin{array}{cc}
                \frac{1}{\l} \left( \coth \o - \frac{\o}{\sinh^2 \o} \right) & \frac{1}{\sqrt{\l}} \left( \frac{\o\cosh \o}{\sinh^2 \o} - \frac{1}{\sinh \o} \right) \\
                 \frac{1}{\sqrt{\l}} \left( \frac{\o\cosh \o}{\sinh^2 \o} - \frac{1}{\sinh \o} \right) & \coth \o - \frac{\o}{\sinh^2 \o}.
            \end{array}  \right).
        \end{equation*}
        We see that its determinant
        \begin{align*}
            \det (\p_k M_{\l,k}) & = \frac{1}{\l} \left\{ \left( \coth \o - \frac{\o}{\sinh^2 \o} \right)^2 - \left( \frac{\o\cosh \o}{\sinh^2 \o} - \frac{1}{\sinh \o} \right)^2 \right\} \\
                                 & = \frac{1}{\l\sinh^4 \o} (\sinh^2 \o - \o^2) (\cosh^2 \o -1) > 0,
        \end{align*}
        as $\sinh \o > \o$ and $\cosh \o >1$ for $\o \in (0,\infty)$. Furthermore, the $(2,2)$-entry of $\p_k M_{\l,k}$ satisfies
        \begin{equation*}
            \coth \o - \frac{\o}{\sinh^2 \o} = \frac{\cosh \o \sinh \o - \o}{\sinh^2 \o} >\frac{\sinh \o - \o}{\sinh^2 \o}> 0 .
        \end{equation*}
        Therefore, by Sylvester's criterion the matrix  $\p_k M_{\l,k}$ is positive definite for $k>0$. 
        Since according to Remark \ref{rem:analytic}, $\m_{k,j}(\l)$ is continuous in $k\in [0,\infty)$ for fixed $\l$, we can conclude $$\m_{k+1,j}(\l) > \m_{k,j}(\l) \quad \text{for all} \quad k\in \Nz, \quad j=1,2.$$ 
	\end{proof}
\end{lemma}

In the final lemma of this section we derive the asymptotics of  $\m_{k,1}(\l)$ and $\m_{k,2}(\l)$ as $k\to\infty$.
\begin{lemma}\label{eigenvalsasymp}
	For fixed $\l \in (0,1)$ the sequences $\{ \m_{k,j}(\l) \}_{k=0}^\infty$, $j=1,2$, have the asymptotics $$\lim_{k\to\infty} \frac{\m_{k,1}(\l)}{k} = 1 , \qquad \lim_{k\to\infty} \frac{\m_{k,2}(\l)}{k} = \frac{1}{\l} .$$
	\begin{proof}
		From the definition of $C$ and $D$ in \eqref{eq:CD} and the fact that $\lim_{k\to\infty} \coth\o = 1$, we calculate 
		$$ \lim_{k\to\infty} \frac{C}{k} = 1 + \l ,\qquad  \lim_{k\to\infty} \frac{D}{k^2} = 1 .$$
		Hence, using equations \eqref{eq:eigvalshift} and \eqref{modeigenvals}, we obtain
		\begin{align*} 
		\lim_{k\to\infty} \frac{\m_{k,j}(\l)}{k} &= \lim_{k\to\infty}\left(\frac{\tilde{\m}_{k,j}(\l)}{k} - \frac{1}{c_\l k}\right) = \lim_{k\to\infty} \frac{1}{2\l} \left( \frac{C}{k} \mp \sqrt{\frac{C^2}{k^2} - \frac{4\l D}{k^2}} \right) \\
		& = \frac{(1+\l) \mp \sqrt{(1+\l)^2 - 4\l}}{2\l} = \left\{\begin{array}{ll}1 & j=1\\ 1/\l  & j=2 \end{array}\right..
		\end{align*}
	\end{proof}
\end{lemma}

As a corollary to the lemmas above, we state the following proposition.
\begin{prop}\label{prop:collectmono} Let $k\in \Nz$ and let $\mu_{k,1}(\l)$ and $\mu_{k,2}(\l)$ be the eigenvalues of the matrix $M_{\l,k}$ defined in Lemma \ref{invariance}. The following statements are satisfied: 
\begin{itemize}
\item both eigenvalues for $k=0$ are negative
\begin{equation}\label{eq:eigk0neg}
\mu_{0,1}(\l) < \mu_{0,2}(\l) < 0 \quad \text{for all} \quad \l\in (0,1);
\end{equation}

\item for $k=1$, the eigenvalues are equal to
	\begin{equation}\label{eq:eigk1}
	\m_{1,1}(\l) = -\frac{1}{c_\l} , \qquad \m_{1,2}(\l) = 0;
	\end{equation} 

\item for every $k\geq 2$, the second eigenvalue
\begin{equation}\label{eq:eig2pos}
\mu_{k,2}(\l) > 0 \quad \text{for all} \quad \l\in (0,1);
\end{equation}

\item for every $k\geq 2$, there exists a unique value $\l_k^*\in (0,1)$ such that the first eigenvalue
\begin{equation}\label{eq:eig1cross}
\mu_{k,1}(\l^*_k)=0.
\end{equation}
Moreover, the sequence $\{\l_k^*\}_{k=2}^{\infty}$ is strictly increasing with $\lim_{k\to \infty}\l_k^* =1$.
\end{itemize}
\begin{proof}

In \eqref{eq:k1evals} we calculated that $\mu_{1,2} \equiv 0$, hence by Lemma \ref{eigenvalsincrease} we have that for $k\geq 2$
\[
\mu_{k,2} > \mu_{1,2} \equiv 0 > \mu_{0,2} > \mu_{0,1}.
\]
so that we show both \eqref{eq:eigk0neg} and \eqref{eq:eig2pos}. Equation \eqref{eq:eigk1} is \eqref{eq:k1evals} reproduced here for the sake of completeness. Only the last bullet point remains to be established.

According to Lemma \ref{eigenvalsdecrease}, for $k\geq 2$ the first branch $\mu_{k,1}(\l)$ is strictly decreasing in $\l$ and
\[
\lim_{\l \to 0} \mu_{k,1}(\l)= k-1 > 0 \quad \text{while}\quad \lim_{\l \to 1} \mu_{k,1}(\l) = -\infty.
\]
Thus, when $k\geq 2$, $\mu_{k,1}(\l)$ has a unique zero $\l=\l_k^*$ in $ (0,1)$, where it changes sign from positive to negative. Since the $k$-monotonicity Lemma \ref{eigenvalsincrease} implies that  
$$
\mu_{k+1,1}(\l_k^*)>\mu_{k,1}(\l_k^*) = 0
$$
we must have $\l_{k+1}^*>\l_k^*$,  so that the sequence of zeros $\{\l_k^*\}_{k=2}^{\infty}$ is strictly increasing. Denote its limit by $l=\lim_{k\to\infty} \l_k^*$. Obviously, $\l_k^*\leq l\leq 1$ for all $k\geq 2$. If it were the case that $l<1$, we would have by the asymptotic behaviour of $\mu_{k,1}(\l)$, established in Lemma \ref{eigenvalsasymp}, that for any large enough $k$, $\mu_{k,1}(l)/k>\frac{1}{2}$. But then the zero $\l_k^*$ of $ \mu_{k,1}(\l)$ would have to be greater than $l$, which is a contradiction. Hence, $l=1$.
\end{proof}
\end{prop}

\section{The proof of Theorem \ref{teo2}}\label{proofcomp}

We now turn to the proof of Theorem \ref{teo2}. Following the discussion given in Section \ref{prelim}, it will be necessary to specialize to functions that are invariant under the action of a subgroup $G$ of the orthogonal group $O(n)$ satisfying (P1)-(P2), stated in Section \ref{prelim}. Recall that $C_{G}^{k,\a}(\S^{n-1})$ denotes the H\"older space of $G$-invariant functions. 

We begin by observing that the operator $F_\l$ defined in \eqref{eq:Flambdapolar} restricts to the $G$-invariant function spaces $(C_{G}^{k,\a}(\S^{n-1}))^2$ and, therefore, so does its linearization $L_\l$.

\begin{lemma}\label{opsrstc}
	The nonlinear operator $F_\l$ defined in \eqref{eq:Flambdapolar} and its linearization $L_\l=D_{\vec{v}}F_\l(\vec{0})$ have well defined restrictions
	$$ \begin{array}{rccc}
		F_\l  : & U & \to & \left( C_{G}^{1,\a}(\S^{n-1}) \right)^2 , \\          
		L_\l  : & \left( C_{G}^{2,\a}(\S^{n-1}) \right)^2 & \to & \left( C_{G}^{1,\a}(\S^{n-1}) \right)^2 ,
	\end{array}	$$
	where $U \subseteq (C_{G}^{2,\a}(\S^{n-1}))^2$ is a sufficiently small neighbourhood of $\vec{0}$. 
	
	\begin{proof}
		We just have to explain why $F_\l(\vec{v}) \in (C_{G}^{1,\a}(\S^{n-1}))^2$ if $\vec{v} \in U \subseteq (C_{G}^{2,\a}(\S^{n-1}))^2$. Clearly, if $\vec{v}$ is $G$-invariant, then so is the pull-back metric $g=g(\vec{v}) = \F^* g_0$ on $\O_\l$, where $\F$ is the diffeomorphism defined in \eqref{eq:diffeo}. Hence, by unique solvability, the solution $u^*_\l(\vec{v})\in C^{2,\a}(\overline{\O}_\l)$ of the Dirichlet problem \eqref{eq:pullbackprob} is also $G$-invariant, and we confirm that $F_\l(\vec{v})$ belongs to $(C^{1,\a}_G(\S^{n-1}))^2$, indeed.
	\end{proof}
\end{lemma}

Recall that properties (P1)-(P2) of $G$ say that the $G$-invariant spherical harmonics are only the ones of degree $\{i_k\}_{k\in \Nz}$, with $i_0 =0$ and $i_1\geq 2$, and for each $k\in \Nz$, they form a one-dimensional subspace -- spanned by the unique $G$-invariant spherical harmonic $Y_k$ of degree $i_k$ and unit $L^2(\S^{n-1})$ norm. For each $k\in \Nz$, let $W_k= \text{Span}\{(Y_k, 0), (0,Y_k)\}$, let $\mathcal{B}_k=\{ \vec{e}_1, \vec{e}_2 \}$ be the orthonormal basis for $W_k$, defined in \eqref{eq:basis}, and let $M_{\l, i_k}$ be the matrix of $L_{\l}|_{W_k}$ with respect to $\mathcal{B}_k$. Also, recall that in Remark \ref{rem:analytic} we chose the eigenvector
\[
v_j(\l, i_k)=(a_{k,j}, b_{k,j}) , \quad k\in \Nz, \; j=1,2, \quad \text{where} \quad a_{k,j}>0 \quad \text{and}\quad  a_{k,j}^2 + b_{k,j}^2 = 1,
\]
to span the eigenspace of $M_{\l, i_k}$, associated with  $\mu_{i_k, j}(\l)$. The corresponding eigenvector of $L_\l$ is
\begin{equation}\label{eq:eigvector}
\vec{z}_{k,j}:= a_{k,j}\vec{e}_1 + b_{k,j}\vec{e}_2 , \qquad k\in \Nz, \quad j=1,2, \quad \text{and its norm} \quad \|\vec{z}_{k,j}\|_\l = 1.
\end{equation}

\begin{remark}\label{orthobasis}
    The sequence of eigenvectors $\{ \vec{z}_{k,j}(\l)\}_{k\in \Nz, j=1,2}$ of $L_\l$ forms an orthonormal basis for the Hilbert space $L_{G}^2(\S^{n-1}) \times L_{G}^2(\S^{n-1})$,  endowed with the inner product $\langle \cdot, \cdot \rangle_\l$ defined in \eqref{innprod}, which is equivalent to the usual one.
\end{remark}

Since $i_1\geq 2$, Proposition \ref{prop:collectmono} says that the eigenvalues $\m_{i_k,1}(\l)$, $k\in \N$, cross $0$ at values $\l_k := \l_{i_k}^* \in (0,1)$, with $\l_k\nearrow 1$, while the eigenvalues $\m_{i_k,2}(\l)>0$. In addition, the eigenvalues $\m_{i_0,1}(\l)$ and $\m_{i_0,2}(\l)$ are strictly negative for all $\l \in (0,1)$.  Theorem \ref{teo2} will follow after a direct application of the Crandall-Rabinowitz Theorem (see Appendix, Theorem \ref{crt}) to the smooth family of nonlinear operators $F_\l: U \to (C_{G}^{1,\a}(\S^{n-1}))^2$ from Lemma \ref{opsrstc}, and the following proposition puts us exactly in the framework of that theorem. In order to simplify notation, we will denote  \[\vec{z}_k := \vec{z}_{k,1} \quad \text{for} \quad k\in \N,\] 
where $\vec{z}_{k,1}$ is defined in \eqref{eq:eigvector}.

\begin{prop}\label{verificationLlambda}
    For every $k \in \N$, the linear operator $L_{\l_k}: (C_{G}^{2,\a}(\S^{n-1}))^2 \to (C_{G}^{1,\a}(\S^{n-1}))^2$ in Lemma \ref{opsrstc} has kernel of dimension 1 spanned by $\vec{z}_k$, closed image of co-dimension 1 given by
    \begin{equation}\label{imLlambda}
        \im L_{\l_k} = \left \{ \vec{w} \in \left( C_{G}^{1,\a}(\S^{n-1}) \right)^2 : \langle \vec{w}, \vec{z}_k \rangle_{\l_k} = 0 \right \} ,
    \end{equation}
    and satisfies
    \begin{equation}\label{transLlambda}
        \left. \frac{\p}{\p\l} L_\l \right|_{\l=\l_k} (\vec{z}_k) \notin \im L_{\l_k} .
    \end{equation}
    
    \begin{proof}
        For a proof of \eqref{imLlambda}, see \cite[Proposition 5.1]{fall2018serrin}, as it follows almost verbatim. For the sake of completeness, we shall provide some of the details. Our first observation is that the Sobolev space $H^s(\S^{n-1})$ can be characterized as the subspace of funtions  $v \in L^2(\S^{n-1})$ such that
        $$ \sum_{j=0}^{\infty} (1+j^2)^s \| P_j (v) \|_{L^2}^2 < \infty ,$$
        where $P_j$ denotes the $L^2$-orthogonal projection on the subspace generated by the spherical harmonics of degree $j$, and we denote $H_{G}^s(\S^{n-1}) := H^s(\S^{n-1}) \cap L_{G}^2(\S^{n-1})$. As stated in Remark \ref{orthobasis}, the sequence $\{ \vec{z}_{k,j} \}$ is an orthonormal basis for $(L_{G}^2(\S^{n-1}))^2$ with inner product $\langle \cdot, \cdot \rangle_\l$, and so we can define the map  $(H_{G}^2(\S^{n-1}))^2 \to (H_{G}^1(\S^{n-1}))^2$
        \begin{equation}\label{eq:Llambdaextension}
            \vec{w} = \sum_{\ell=0}^\infty ( a_{\ell,1} \vec{z}_{\ell,1} + a_{\ell,2} \vec{z}_{\ell,2} ) \mapsto \sum_{\ell=0}^\infty ( a_{\ell,1} \m_{i_\ell,1}(\l) \vec{z}_{\ell,1} + a_{\ell,2} \m_{i_\ell,2}(\l) \vec{z}_{\ell,2} ) .
        \end{equation}
         Due to the asymptotic behavior of the sequences $\{ \m_{m,j}(\l) \}_{m=1}^\infty$ proved in Lemma \ref{eigenvalsasymp}, we can see that \eqref{eq:Llambdaextension} defines a continuous mapping. Since it agrees with $L_\l$ on finite linear combinations of $\{ \vec{z}_{k,j} \}$, which are dense both in $(C_{G}^{2,\a}(\S^{n-1}))^2$ and $(H_{G}^2(\S^{n-1}))^2$, \eqref{eq:Llambdaextension} defines an extension of $L_\l$. Moreover, for $\l=\l_k$ the map
        \begin{equation*}
            \vec{w} = \sum_{\ell=0}^\infty ( b_{\ell,1} \vec{z}_{\ell,1} + b_{\ell,2} \vec{z}_{\ell,2} ) \mapsto \sum_{\substack{\ell=0 \\ \ell \neq k}}^{\infty} \left( \frac{b_{\ell,1}}{\m_{i_\ell,1}(\l_k)} \vec{z}_{\ell,1} + \frac{b_{\ell,2}}{\m_{i_\ell,2}(\l_k)} \vec{z}_{\ell,2} \right) + \frac{b_{k,2}}{\m_{i_k,2}(\l_k)} \vec{z}_{k,2}
        \end{equation*}
        is a right inverse for $L_{\l_k}$, which is also continuous by Lemma \ref{eigenvalsasymp}. Thus, $L_{\l_k}$ defines an isomorphism between the spaces
        \begin{align*}
            \mathfrak{X}_k & := \left\{ \vec{v} \in \left( H_{G}^2(\S^{n-1}) \right)^2 : \langle \vec{v}, \vec{z}_k \rangle_{\l_k} = 0 \right\} ,  \\
            \mathfrak{Y}_k & := \left\{ \vec{v} \in \left( H_{G}^1(\S^{n-1}) \right)^2 : \langle \vec{v}, \vec{z}_k \rangle_{\l_k} = 0 \right\} .
        \end{align*}
        It follows that $L_{\l_k}: \mathfrak{X}_k \cap (C_{G}^{2,\a}(\S^{n-1}))^2 \to \mathfrak{Y}_k \cap (C_{G}^{1,\a}(\S^{n-1}))^2$ is a well defined, injective mapping. It only remains to prove its surjectivity. For that purpose, let $\vec{y} \in \mathfrak{Y}_k \cap (C_{G}^{1,\a}(\S^{n-1}))^2$. Then there exists $\vec{w} \in \mathfrak{X}_k$ such that $L_{\l_k}(\vec{w}) = \vec{y}$. The latter means that the weak solution $\f \in H^2(\O_\l)$ to
        \begin{equation*}\begin{cases}
            \D \f = 0 & \text{in } \O_\l , \\
            \f = w_1 & \text{on } \G_\l , \\
            \f = w_2 & \text{on } \G_1 ,
            \end{cases}
        \end{equation*}
        satisfies
                \begin{align*}
            -\frac{\p\f}{\p\n} + \frac{w_1}{c_\l} \frac{\p_r^2 u_\l}{\p r^2} & = y_1 \quad \text{on } \G_\l , \\
            -\frac{\p\f}{\p\n} + \frac{w_2}{c_\l} \frac{\p_r^2 u_\l}{\p r^2} & = y_2  \quad \text{on } \G_1 .
        \end{align*}
        
        From here  one argues that $\f \in W^{2,p}(\O_\l)$ for every $p \in (1,\infty)$ so that by Sobolev embedding $\f \in C^{1,\a}(\overline{\O}_\l)$, for all $0<\a<1$. But then $\phi$ is also weak solution to  the Neumann problem
              \begin{equation*}\begin{cases}
            \D \f = 0 & \text{in } \O_\l , \\
             \frac{\p\f}{\p\n} =\frac{w_1}{c_\l} \frac{\p_r^2 u_\l}{\p r^2}- y_1 & \text{on } \G_\l , \\
            \frac{\p\f}{\p\n} = \frac{w_2}{c_\l} \frac{\p_r^2 u_\l}{\p r^2} - y_2 & \text{on } \G_1 ,
 \end{cases}
        \end{equation*}
  with Neumann conditions in $C^{1,\a}$. Hence, $\phi \in C^{2,\a}(\overline{\O}_\l)$,    
         which implies $\vec{w} \in (C_{G}^{2,\a}(\S^{n-1}))^2$.
        
        Therefore, $L_{\l_k}: \mathfrak{X}_k \cap (C_{G}^{2,\a}(\S^{n-1}))^2 \to \mathfrak{Y}_k \cap (C_{G}^{1,\a}(\S^{n-1}))^2$ is also an isomorphism. This readily implies equality \eqref{imLlambda} and that $\ker L_{\l_k}$ is spanned by $\vec{z}_k$. 
        The tranversality condition \eqref{transLlambda} follows from the fact that
        $$ \frac{\p}{\p\l} L_\l \bigg|_{\l=\l_k} (\vec{z}_k) = \mu_{i_k,1}'(\l_k) \vec{z}_k .$$
        which by Lemma \ref{eigenvalsdecrease} is a non-zero scalar multiple of $\vec{z}_k$.
    \end{proof}
\end{prop}

\begin{proof}[Proof of Theorem \ref{teo2}]
Let $I_k\subseteq (0,1)$ be a small interval around each critical value $\l_k\nearrow 1$. Let $U_k\subseteq (C_{G}^{2,\a}(\S^{n-1}))^2$ be an appropriately small neighbourhood of $\vec{0}$, such that for all $\l\in I_k$, $F_\l$ is well defined on $U_k$  via \eqref{eq:Flambdapolar}. Then the operator 
    \[F: U_k \times I_k \to  Y:=(C_{G}^{1,\a}(\S^{n-1}))^2, \qquad F(\vec{v}, \l):= F_\l(\vec{v}) ,\] 
is in $C^{\infty}(U_k \times I_k , Y)$, and by Proposition \ref{verificationLlambda}, we can apply the Crandall-Rabinowitz Bifurcation Theorem to get a smooth curve
    $$\begin{array}{ccc}
    	(-\e,\e) & \to & \left( C_{G}^{2,\a}(\S^{n-1}) \right)^2 \times I_k \\
        s & \mapsto & (\vec{w}(s),\l(s))
    \end{array}$$
    such that
    \begin{itemize}
        \item $\vec{w}(0) =\vec{0}$, $\l(0) = \l_k$, and $\langle \vec{w}(s), \vec{z}_k \rangle_{\l_k} =0$; 
        \item $F_{\l(s)}(\vec{v}(s)) = 0$, where $\vec{v}(s) = s(\vec{z}_k + \vec{w}(s))$.
    \end{itemize}
    Then, for every $s \in (-\e,\e)$, the solution $u_{\l(s)}(\vec{v}(s)) \in C_{G}^{2,\a}(\overline{\O}_{\l(s)}^{\vec{v}(s)})$  to the Dirichlet problem \eqref{eq:prob2} also solves the overdetermined problem \eqref{eq:prob1}. 
\end{proof}

\section{Proof of Corollary \ref{coro:cheeger}}\label{sec:coro}

Let $\O$ be any one of the domains constucted in Theorem \ref{teo1} and let $u\in C^{\infty}(\overline{\O})$ be the solution of the corresponding overdetermined problem
\begin{equation}\label{eq:prob1.1}
\begin{cases}
    -\D u = 1   & \text{ in } \O , \\
    u = 0 & \text{ on } \p\O_0 , \\
    u = a & \text{ on } \p\O_1 , \\
    \p_\n u = c &\text { on } \p \O.
\end{cases}
\end{equation}
for some constants $a>0$ and $c>0$.

\begin{proof}[Proof of Corollary \ref{coro:cheeger}]
First, let us show that 
\begin{equation}\label{eq:gradientbound}
|\nabla u| < c \quad \text{in} \quad \O.
\end{equation}
Indeed, since $-\D u =1$, 
\[
\Delta |\nabla u|^2 = 2 |D^2 u|^2 + 2 \nabla u \cdot \nabla (\D u) =  2 |D^2 u|^2 >0,
\]
so that the function $|\nabla u|^2$ is subharmonic in $\O$ and, by the strict maximum principle 
\[
|\nabla u|^2 (x) < \sup_{\de \O}|\nabla u|^2 = c^2 \quad \text{for all}\quad x\in \O.
\]
Now let $E\subseteq \O$ be any subset of finite perimeter and let $\de^* E\subseteq \de E$ be its reduced boundary -- where one can define a measure-theoretic inner unit normal $\nu_E$. By De Giorgi's theorem (see \cite{giusti1984}), the $(n-1)$-Haudorff dimensional measure $H^{n-1}(\de^* E) = P(E)$ and we can apply the version of the Divergence Theorem to obtain
\begin{equation}\label{eq:div}
|E| = \int_E (-\D u) \, dx = \int_{\de^* E} \nabla u \cdot \nu_E \; d H^{n-1} \leq c H^{n-1}(\de^* E) = c P(E),
\end{equation}
where we used \eqref{eq:gradientbound} in the inequality above. Hence, $P(E)/E \geq 1/c$ and equality holds if and only if $E=\O$.

\end{proof}

\section{Appendix}

We give here a version of the Crandall-Rabinowitz Theorem, equivalent to the one stated on \cite{crandall1971bifurcation}, which is the principal tool behind Theorem \ref{teo2}. For a proof of the theorem and applications, we refer the reader to \cite{crandall1971bifurcation,kielhofer2011bifurcation}.

\begin{theorem}[Crandall-Rabinowitz]\label{crt}
    Let $X$ and $Y$ be Banach spaces, and let $U \subset X$ and $I \subset \R$ be open sets, such that $0 \in U$. Let $F \in C^2 (U \times I , Y)$ and assume
    
    \begin{enumerate}
        \item $F(0,\l) = 0$ for all $\l \in I$;
        \item $\ker \p_x F(0,\l_0)$ is a dimension $1$ subspace and $\im \p_x F(0,\l_0)$ is a closed co-dimension $1$ subspace for some $\l_0 \in I$;
        \item $\p_\l \p_x F(0,\l_0)(x_0) \notin \im \p_x F(0,\l_0)$, where $x_0 \in X$ spans $\ker \p_x F(0,\l_0)$.
    \end{enumerate}
    
    Write $X = \hat{X} \oplus \R x_0$. Then there exists a $C^1$ curve
    $$ (-\e,\e) \to \hat{X} \times \R \;, \qquad s \mapsto (x(s),\l(s)) $$
    such that
    
    \begin{itemize}
        \item $x(0) = 0$ and $\l(0) = \l_0$;
        \item $s(x_0+x(s)) \in U$ and $\l(s) \in I$;
        \item $F(s(x_0+x(s)),\l(s)) = 0$.
    \end{itemize}
    
    Moreover, there is a neighbourhood of $(0,\l_0)$ such that $\{ (s(x_0+x(s)),\l(s)) : s \in (-\e,\e) \}$ are the only solutions bifurcating from $\{ (0,\l) : \l \in I \}$.
\end{theorem}

\bibliography{OD}

\end{document}